\documentclass[11pt]{article}%
\usepackage{color}
\usepackage{amssymb, graphicx, amsmath}
\usepackage{amsmath,amsthm,amssymb}
\usepackage{booktabs}
\usepackage{cases}
\usepackage{subfigure}
\usepackage{graphicx}
\usepackage{CJK,graphicx}
\usepackage[figuresright]{rotating}

\newcommand{\nn}{\nonumber}

\usepackage{multirow}
\usepackage{listings}
\usepackage{float}
\usepackage{adjustbox}
\usepackage{pdflscape}
\usepackage{epstopdf}
\usepackage{rotating}
\usepackage{url}
\usepackage{hyperref}
\usepackage{diagbox}

\definecolor{myblue}{rgb}{0,0,0.5}
\definecolor{mygreen}{rgb}{0,0.5,0}
\definecolor{myred}{rgb}{0.5,0,0}
\usepackage{ulem}
\usepackage{textcomp}

\newtheorem{theorem}{Theorem}[section]

\newtheorem{lemma}{Lemma}[section]

\newtheorem{remark}{Remark}

\newtheorem{property}{Property}[section]

\makeatletter
\def\hlinew#1{%
  \noalign{\ifnum0=`}\fi\hrule \@height #1 \futurelet
   \reserved@a\@xhline}

\def \[{\begin{equation}}
\def \]{\end{equation}}
\begin{document}
\begin{CJK*}{GBK}{song}

\begin{center}

{\Large \bf  A rank-two relaxed parallel splitting version of the augmented Lagrangian method with step size in $(0,2)$ for separable convex programming}\\

\bigskip
\medskip

 {\bf Bingsheng He}\footnote{\parbox[t]{16cm}{
 Department of Mathematics,  Nanjing University, China.
  This author was supported by the NSFC Grant 11871029. Email: hebma@nju.edu.cn}}
 \quad
  {\bf Feng Ma}\footnote{\parbox[t]{16.0cm}{
High-Tech Institute of Xi'an, Xi'an, 710025, Shaanxi, China. This author was supported by the NSFC Grant 12171481.  Email: mafengnju@gmail.com}}
   \quad
 {\bf Shengjie Xu}\footnote{\parbox[t]{16cm}{
 Department of Mathematics, Harbin Institute of Technology, Harbin, China, and Department of Mathematics,   Southern  University of Science and Technology, Shenzhen, China. This author was supported by the NSFC Grant 11871264. Email: xsjnsu@163.com
  }}
  \quad
 {\bf Xiaoming Yuan}\footnote{\parbox[t]{16cm}{
 Department of Mathematics, The University of Hong Kong, Hong Kong. This author was supported  by a URC Supplementary Funding for Faculties/Units of Assessment from HKU. Email:  xmyuan@hku.hk
  }}

\bigskip

\today

\end{center}

{\small

\parbox{0.95\hsize}{

\hrule

\medskip

{\bf Abstract.} The augmented Lagrangian method (ALM) is classic for canonical convex programming problems with linear constraints, and it finds many applications in various scientific computing areas. A major advantage of the ALM is that the step for updating the dual variable can be further relaxed with a step size in $(0,2)$, and this advantage can easily lead to numerical acceleration for the ALM. When a separable convex programming problem is discussed and a corresponding splitting version of the classic ALM is considered, convergence may not be guaranteed and thus it is seemingly impossible that a step size in $(0,2)$ can be carried on to the relaxation step for updating the dual variable. We show that for a parallel splitting version of the ALM, a step size in $(0,2)$ can be maintained for further relaxing both the primal and dual variables if the relaxation step is simply corrected by a rank-two matrix. Hence, a rank-two relaxed parallel splitting version of the ALM with a step size in $(0,2)$ is proposed for separable convex programming problems. We validate that the new algorithm can numerically outperform existing algorithms of the same kind significantly by testing some applications.

\medskip

 \medskip

  \hrule

  }}

\bigskip


\section{Introduction}  \label{Introd}

We start with the canonical convex programming problem with linear equality constraints:
\begin{equation}\label{problem}
  \min\big\{\theta(x)  \mid  \mathcal{A}x=b,  \;  x\in\mathcal{X} \big\},
\end{equation}
where $\theta:\Re^{n}\to {\Re}$ is a closed, proper and convex but not necessarily smooth function, $\mathcal{X}\subseteq\Re^n$ is a closed convex set, ${\cal {A}}\in\Re^{m\times n}$, and $b\in\Re^m$. A classic algorithm for solving (\ref{problem}) is the augmented Lagrangian method (ALM) which was introduced in \cite{Hes69} and \cite{Powell69}, individually. Let
$$ \mathcal{L}_{\beta}(x, \lambda):=\theta(x)-\lambda^T(\mathcal{A}x-b) + \frac{\beta}{2}\|\mathcal{A}x-b\|_2^2$$
be the augmented  Lagrangian function of \eqref{problem} with $\lambda\in\Re^m$ the Lagrange multiplier and $\beta>0$ the penalty parameter. Then, with given $\lambda^k\in\Re^m$, the ALM generates a new iterate $(x^{k+1},\lambda^{k+1})$ via the scheme
\begin{subequations} \label{ALM}
\begin{numcases}{\hbox{(ALM )}}
\label{ALM-x} x^{k+1} \;\in\; \arg\min\big\{ \mathcal{L}_{\beta}(x, \lambda^k)   \mid  x\in\mathcal{X}\big\},\\[0.1cm]
\label{ALM-y} \lambda^{k+1} \;=\; \lambda^k-\beta(\mathcal{A}x^{k+1}-b).
\end{numcases}
\end{subequations}
Hereafter, we also call $x$ and $\lambda$ the primal and dual variables, respectively. Since 1970s, the ALM has constantly found many applications in various scientific computing areas such as optimal control, image processing, optimization, and machine learning. We refer to, e.g., \cite{ABMS-1,ABMS-2,Bert1996,Birgin2014,CP-Acta,Fortin1983,Glowinski89,Ito2008}, for some monographs and papers about the ALM.

As analyzed in \cite{Rock76B}, the ALM \eqref{ALM} can be interpreted as an application of the proximal point algorithm (PPA) \cite{Mar70} to the dual problem of \eqref{problem}. This PPA perspective has immediately inspired the following relaxation version of the ALM \eqref{ALM} in \cite{Gol1979}:
 \begin{subequations} \label{Relax}
\begin{numcases}{\hbox{(Relaxed ALM )}}
\label{RALM-x}x^{k+1} \;\in\;  \arg\min\big\{ \mathcal{L}_{\beta}(x, \lambda^k)   \;|\;  x\in\mathcal{X}\big\}, \\[0.1cm]
\label{RALM-y}\;\;\tilde{\lambda}^k \;\;\,=\;  \lambda^k-\beta(\mathcal{A}x^{k+1}-b), \\[0.1cm]
\label{RALM-R}\lambda^{k+1} \;=\;  \lambda^k-\alpha(\lambda^k-\tilde{\lambda}^k),
\end{numcases}
\end{subequations}
in which the update for the dual variable $\lambda$ is further relaxed with the step size $\alpha\in(0,2)$. It is known that the further relaxation step (\ref{RALM-R}) can easily accelerate the convergence of the classic ALM \eqref{ALM}, as analyzed in \cite{Eck92,Fortin1983,Glowinski89,TaoYuan2018} and shown empirically in \cite{Bert1996,Eck94,FHLY, Glowinski89,Wen}. The validation of further relaxing the update for the dual variable with a large step size in $(0,2)$ is indeed a major advantage of the ALM (\ref{ALM}). When the canonical model (\ref{problem}) is more specific, usually the ALM (\ref{ALM}) should also be adapted to more specific and practical forms for being implemented. If the model (\ref{problem}) has separable structure and thus the ALM (\ref{ALM}) is decomposed conformally, more precisely, if the $x$-subproblem \eqref{ALM-x} is decomposed into multiple ones with respect to all $x_i$'s, can the relaxation step \eqref{RALM-R} with $\alpha\in (0,2)$ be still valid? We mainly focus on answering this question in this paper.

Let us nail down the question with details. Consider the following multiple-block separable convex programming problem whose objective function is the sum of multiple subfunctions without coupled variables:
\begin{equation}\label{A-Problem-M}
\min\Big\{\sum_{i=1}^p \theta_i(x_i) \mid \sum_{i=1}^p A_i x_i =b; \; x_i\in {\cal X}_i, \; i=1,\ldots, p\Big\},
\end{equation}
where $\theta_i: {\Re}^{n_i}\to {\Re} \;(i=1,\ldots, p)$ are all closed, proper and convex but not necessarily smooth functions, ${\cal X}_i\subseteq \Re^{n_i} \;(i=1,\ldots, p)$ are closed convex sets,
$A_i\in \Re^{m\times n_i} \;(i=1,\ldots, p)$, $b\in \Re^m$, and $p\ge 2$. We refer to, e.g.,  \cite{boyd2010distributed,chandrasekaran2012latent,kiwiel1999proximal,McLachlan,PGWM,tao2011recovering}, for various applications that can be formed as the separable convex programming model \eqref{A-Problem-M}. Throughout our discussion, the solution set of \eqref{A-Problem-M} is assumed to be nonempty, and each $A_i$ in \eqref{A-Problem-M} is assumed to be full column-rank. The separable model \eqref{A-Problem-M} can be regarded as a special case of the canonical model \eqref{problem} with
\begin{equation}\label{notations}
x=\left(\!\!
\begin{array}{c}
x_1 \\[-0.15cm]
\vdots \\[-0.2cm]
x_p \\
\end{array}\!\!
\right), \;\; \theta(x)=\sum_{i=1}^p \theta_i(x_i), \;\;  \mathcal{A}=
\left[\!\!
\begin{array}{ccc}
A_1 & \cdots & A_p \\
\end{array}\!\!
\right], \;\; \mathcal{X}=\mathcal{X}_1\times\cdots\times\mathcal{X}_p ,\;\; \hbox{and} \;\; n=\sum_{i=1}^pn_i.
\end{equation}
If the separable structure is overlooked and the ALM \eqref{ALM} is applied directly to \eqref{A-Problem-M}, the resulting scheme reads as
\begin{subequations} \label{A-ALM}
\begin{numcases}{}
\label{A-ALM-x} (x_1^{k+1}, \ldots, x_p^{k+1}) \;\in \; \arg\min\bigl\{{\cal L}_{\beta}(x_1, \ldots, x_p, \lambda^k)   \mid  x_i\in \mathcal{X}_i,\; i=1,\ldots,p\bigr\},\\[0.1cm]
\label{A-ALM-y} \qquad \qquad \quad  \lambda^{k+1} \;=\; \lambda^k- \beta (\textstyle\sum_{i=1}^p A_ix_i^{k+1} -b),
\end{numcases}
\end{subequations}
where $ \mathcal{L}_{\beta}(x_1,\ldots, x_p, \lambda)$ is the augmented Lagrangian function of (\ref{A-Problem-M}) defined as
\begin{equation}\label{LagranM}
  \mathcal{L}_{\beta}(x_1,\ldots, x_p, \lambda):=\sum_{i=1}^p\theta_i(x_i)-\lambda^T(\sum_{i=1}^pA_ix_i-b)+\frac{\beta}{2}\|\sum_{i=1}^pA_ix_i-b\|_2^2
\end{equation}
with the Lagrange multiplier $\lambda \in\Re^m$ and the penalty parameter $\beta>0$. To alleviate the subproblem \eqref{A-ALM-x} and to effectively exploit the separable structure of the model \eqref{A-Problem-M}, it is natural to consider decomposing the subproblem \eqref{A-ALM-x} as $p$ subproblems with respect to $x_i$'s by either the Gauss-Seidel or the Jacobian manner. More specifically, decomposing the subproblem \eqref{A-ALM-x} by the Gauss-Seidel manner yields the scheme
\begin{equation}\label{DADMM}
\left\{
  \begin{array}{lcl}
        x_1^{k+1}  & \in & \arg\min\big\{ \mathcal{L}_{\beta}(x_1,x_2^k,\ldots, x_p^k, \lambda^k)  \mid  x_1\in\mathcal{X}_1\big\},\\
                           &    & \qquad \vdots  \\
        x_i^{k+1} & \in & \arg\min\big\{ \mathcal{L}_{\beta}(x_1^{k+1},\ldots,x_{i-1}^{k+1},x_i,x_{i+1}^k,\ldots, x_p^k, \lambda^k)  \mid x_i\in\mathcal{X}_i\big\},\\
           &  &\qquad \vdots \\
        x_p^{k+1} & \in &\arg\min\big\{ \mathcal{L}_{\beta}(x_1^{k+1},\ldots, x_{p-1}^{k+1},x_p, \lambda^k)  \mid  x_p\in\mathcal{X}_p\big\},\\[0.20cm]
 \lambda^{k+1} & = & \lambda^k-\beta (\sum_{i=1}^p A_ix_i^{k+1}-b);
 \end{array}
\right.
\end{equation}
and the resulting scheme by decomposing the subproblem \eqref{A-ALM-x} by the Jacobian manner is
\begin{equation}\label{PSALM-P}
  \left\{
  \begin{array}{lcl}
        x_1^{k+1}  & \in & \arg\min\big\{ \mathcal{L}_{\beta}(x_1,x_2^k,\ldots, x_p^k, \lambda^k) \mid x_1\in\mathcal{X}_1\big\},\\
                           &    & \qquad \vdots  \\
        x_i^{k+1} & \in & \arg\min\big\{ \mathcal{L}_{\beta}(x_1^{k},\ldots,x_{i-1}^k,x_i,x_{i+1}^k,\ldots, x_p^k, \lambda^k)  \mid x_i\in\mathcal{X}_i\big\},\\
           & & \qquad \vdots  \\
        x_p^{k+1}  & \in & \arg\min\big\{ \mathcal{L}_{\beta}(x_1^{k},\ldots, x_{p-1}^{k},x_p, \lambda^k) \mid x_p\in\mathcal{X}_p\big\},\\[0.20cm]
       \lambda^{k+1} & = & \lambda^k-\beta (\sum_{i=1}^p A_ix_i^{k+1}-b).
  \end{array}
\right.
\end{equation}
Despite the well studied convergence of the ALM (\ref{ALM}), or (\ref{A-ALM}), it is not true that the splitting versions (\ref{DADMM}) and (\ref{PSALM-P}) must also be convergent. Indeed, without further assumptions and/or conditions on the functions, coefficient matrices, and the penalty parameter $\beta$, only the case of (\ref{DADMM}) with $p=2$ is convergent, which is well known as the alternating direction method of multipliers (ADMM) studied originally in \cite{GM}. The divergence of (\ref{DADMM}) with $p\ge 3$ and the divergence of (\ref{PSALM-P}) with $p\ge 2$ have been shown in \cite{CHYY} and \cite{HeHouYuanSIOPT}, respectively. Hence, the lack of convergence guarantee imposes the impossibility of carrying on the relaxation step \eqref{RALM-R} to (\ref{DADMM}) and (\ref{PSALM-P}) for the generic setting of the separable convex programming problem (\ref{A-Problem-M}). It is worth noting that the only convergence-guaranteeing case of (\ref{DADMM}) with $p=2$, i.e., the ADMM, the step size for further relaxing the update of the dual variable $\lambda$ can be in $(0, (\sqrt{5}+1)/2)$, as shown in \cite{Fortin1983}. It was asked by Glowinski in \cite{Glow84} if the step size could be enlarged to $(0,2)$. This question is still open for the generic case of (\ref{A-Problem-M}), while an affirmative answer was given in \cite{TY-JOTA} only for the special case of (\ref{A-Problem-M}) with $p=2$ and both subfunctions in the objective are quadratic.

In this paper, we investigate how to relax the parallel splitting version of the ALM (\ref{PSALM-P}), whose convergence is not guaranteed, such that the convergence can be guaranteed while the step size in $(0,2)$ can be maintained for the relaxation step. There is an earlier such effort  \cite{HeHouYuanSIOPT}, in which the output of (\ref{PSALM-P}) is further relaxed by the step
\begin{equation}\label{PSALM-C}
  w^{k+1}:=w^k-\alpha(w^k-w^{k+1}),
\end{equation}
where $w=(x_1;\ldots;x_p;\lambda)$, and
\begin{equation}\label{PSALM-step}
\alpha\in \big(0,2\big(1- \sqrt{p/(p+1)}\big)\big).
\end{equation}
For the relaxation step (\ref{PSALM-C}), both the primal and dual variables are further relaxed, but the step size $\alpha$ shrinks to $0$ when $p$ is large. Hence, the relaxation step (\ref{PSALM-C}) becomes ineffective and numerical acceleration can be hardly expected via the relaxation step (\ref{PSALM-C}) if $p$ is too large. Another relevant work is \cite{HXY2016} (see also \cite{HeTaoYuanIMA} and \cite{Deng2017}), in which there is no further relaxation step but a proximal regularization term is suggested to be added to the objective function of each $x_i$-subproblem. The resulting scheme is
\begin{equation}\label{PJALMP}
  \left\{
  \begin{array}{lcl}
   x_1^{k+1} & \in & \arg\min\big\{ \mathcal{L}_{\beta}(x_1,x_2^k,\ldots, x_p^k, \lambda^k) +\frac{\tau\beta}{2} \|A_1x_1-A_1x_1^k\|^2 \mid x_1\in\mathcal{X}_1\big\},\\
                      &    & \qquad \vdots  \\
   x_i^{k+1} & \in & \arg\min\big\{ \mathcal{L}_{\beta}(x_1^{k},\ldots,x_{i-1}^k,x_i,x_{i+1}^k,\ldots, x_p^k, \lambda^k)+\frac{\tau\beta}{2} \|A_ix_i-A_ix_i^k\|^2  \mid x_i\in\mathcal{X}_i\big\},\\
                      &    & \qquad \vdots  \\
   x_p^{k+1} & \in & \arg\min\big\{ \mathcal{L}_{\beta}(x_1^{k},\ldots, x_{p-1}^{k},x_p, \lambda^k) +\frac{\tau\beta}{2} \|A_px_p-A_px_p^k\|^2 \mid x_p\in\mathcal{X}_p\big\},\\[0.3cm]
   \lambda^{k+1} & = & \lambda^k-\beta (\sum_{i=1}^p A_ix_i^{k+1}-b).
  \end{array}
\right.
\end{equation}
To ensure the convergence of (\ref{PJALMP}), the regularization parameter $\tau$ is required to satisfy the condition
\begin{equation}\label{PJALMC}
 \tau > p-1.
\end{equation}
Recently, it was shown in \cite{Hemyuan2020} that the condition \eqref{PJALMC} can be optimally improved as $\tau>0.75p-1$.
Hence, if $p$ is large, the regularization terms in (\ref{PJALMP}) dominate the objective functions of all $x_i$-subproblems. Accordingly, tiny step sizes are inevitably generated and the convergence will be slowed down. It can be seen that all the existing works \cite{Deng2017,HeHouYuanSIOPT,Hemyuan2020,HeTaoYuanIMA,HXY2016} essentially require additional $p$-dependent conditions to ensure the convergence, whilst these conditions do not favor ensuring the favorable advantage of maintaining a large step size in $(0,2)$ for updating the primal and/or dual variables.

As mentioned, our purpose is to discuss how to maintain a step size in $(0,2)$ for further relaxing the primal and/or dual variables for the parallel splitting version of the ALM (\ref{PSALM-P}), whilst the difficulty of the subproblems in (\ref{PSALM-P}) is also maintained. In a nutshell, our idea is a synergetic combination of \eqref{PSALM-P}-\eqref{PSALM-C} and (\ref{PJALMP}), but their respective shortages in (\ref{PSALM-step}) and (\ref{PJALMC}) are all overcame. More specifically, we keep the proximal regularization for the $x_i$-subproblems as (\ref{PJALMP}) but remove $\tau$ from the coefficients of the proximal terms, and we keep a relaxation step similar as (\ref{PSALM-C}) for further relaxing both the primal and dual variables but with a step size in $(0,2)$, rather than (\ref{PSALM-step}). The new algorithm does not have any additional condition depending on the number of blocks $p$ such as (\ref{PSALM-step}) and (\ref{PJALMC}), while its subproblems are analytically of the same difficulty as those in (\ref{DADMM}), (\ref{PSALM-P}) and \eqref{PJALMP}. This seems to be the first algorithm stemming from the idea of splitting the ALM (\ref{A-ALM}) for the generic setting of the separable convex programming problem (\ref{A-Problem-M}), while the convergence can be theoretically guaranteed and a step size in $(0,2)$ can be maintained for relaxing both the primal and dual variables.


The rest of the paper is organized as follows.
In Section \ref{sec2}, we summarize some preliminaries, introduce some matrices and notations, and prove some elementary assertions for streamlining our analysis. Then, we present the new algorithm in Section \ref{sec3}, and give some remarks. In Section \ref{sec4}, convergence of the new algorithm is analyzed; and convergence rates of the new algorithm in both the ergodic and point-wise senses are derived in Section \ref{Sec-rate}. Some numerical results are reported in Section \ref{sec5}. We briefly discuss some extensions in Section \ref{secE}, and finally some conclusions are drawn in Section \ref{sec6}.


\section{Preliminaries}\label{sec2}
\setcounter{equation}{0}
\setcounter{remark}{0}
\setcounter{proposition}{0}

In this section, we summarize some preliminaries, define some basic matrices and notations for simplifying the presentation of further analysis, and prove some elementary assertions. We first present a fundamental lemma which will be frequently used in our analysis. Its elementary proof can be found in, e.g., \cite{Beck}.
\begin{lemma} \label{CP-TF}
\begin{subequations} \label{CP-TF0}
 Let ${\cal Z} \subseteq \Re^l$ be a closed convex set, $\theta: \Re^l\rightarrow \Re$ and $f : \Re^l\rightarrow \Re$ be convex functions. If $f$ is differentiable on an open set which contains ${\cal Z}$, and the solution set of the minimization problem
 $$\min\{\theta(z) + f(z) \mid z\in {\cal Z}\}$$ is nonempty, then we have
 \begin{equation}\label{CP-TF1}
   z^\ast  \in \arg\min \{  \theta(z) + f(z)  \mid  z\in \mathcal{Z}\}
 \end{equation}
if and only if
 \begin{equation}\label{CP-TF2}
   z^\ast\in \mathcal{Z}, \quad   \theta(z) - \theta(z^\ast) + (z-z^\ast)^T\nabla f(z^\ast) \ge 0, \quad \forall\, z\in \mathcal{Z}.
 \end{equation}
\end{subequations}
\end{lemma}

\subsection{Variational inequality characterization}
Similar as our previous works such as \cite{HeHouYuanSIOPT,he2012alternating}, our analysis will be conducted in the variational inequality (VI) context.
Let us first characterize the corresponding VI of the optimality condition for the model \eqref{A-Problem-M}. Note that the Lagrangian function of \eqref{A-Problem-M} is
\begin{equation}\label{lagra}
  L(x_1,\ldots, x_p, \lambda)=\sum_{i=1}^p\theta_i(x_i)-\lambda^T(\sum_{i=1}^pA_ix_i-b),
\end{equation}
which is defined on the set $\Omega:=\mathcal{X}_1\times\cdots\times\mathcal{X}_p\times\Re^m$ and  $\lambda\in\Re^m$ is the Lagrange multiplier. We call $(x_1^\ast,\ldots,x_p^\ast,\lambda^\ast)\in\Omega$ a saddle point of \eqref{lagra} if it satisfies the following inequalities:
$$L(x_1^\ast,\ldots,x_p^\ast,\lambda)\leq L(x_1^\ast,\ldots,x_p^\ast,\lambda^\ast)\leq L(x_1,\ldots,x_p,\lambda^\ast), \quad \forall \; (x_1,\ldots,x_p,\lambda)\in\Omega.$$
Moreover, such a saddle point $(x_1^\ast,\ldots,x_p^\ast,\lambda^\ast)\in\Omega$ can be also characterized by
\begin{equation*}
  \left\{
    \begin{array}{ll}
     x_1^\ast \in \arg\min\big\{L(x_1,x_2^\ast,\ldots,x_p^\ast,\lambda^\ast) \mid x_1\in \mathcal{X}_1\big\}, \\[-0.1cm]
         \qquad \qquad   \vdots \\[-0.1cm]
     x_i^\ast \in \arg\min\big\{L(x_1^\ast,\ldots,x_{i-1}^\ast,x_i,x_{i+1}^\ast,\ldots,x_p^\ast,\lambda^\ast) \mid x_i\in \mathcal{X}_i\big\}, \\[-0.1cm]
     \qquad \qquad   \vdots \\[-0.1cm]
     x_p^\ast \in \arg\min\big\{L(x_1^\ast,\ldots,x_{p-1}^\ast, x_p,\lambda^\ast) \mid x_p\in \mathcal{X}_p\big\}, \\[0.2cm]
     \lambda^\ast \in \arg\max\big\{L(x_1^\ast,\ldots,x_p^\ast,\lambda) \mid \lambda\in \Re^m\big\}.
    \end{array}
  \right.
\end{equation*}
Recall that the subfunctions $\theta_i \,(i=1,\ldots,p)$ are not assumed to be smooth in \eqref{A-Problem-M}. According to Lemma \ref{CP-TF}, $(x_1^\ast,\ldots,x_p^\ast,\lambda^\ast)$ also satisfies the following inequalities:
\begin{equation}\label{VI}
\left\{ \begin{array}{rl}
\theta_1(x_1) - \theta_1(x_1^\ast) + (x_1-x_1^\ast)^T(- A_1^T\lambda^\ast) \ge 0, & \forall\; x_1\in {\cal X}_1, \\[-0.1cm]
  \vdots  \qquad \qquad \qquad \qquad & \\[-0.1cm]
\theta_i(x_i) - \theta_i(x_i^\ast) + (x_i-x_i^\ast)^T(- A_i^T\lambda^\ast) \ge 0, & \forall\; x_i\in {\cal X}_i, \\[-0.1cm]
  \vdots  \qquad \qquad \qquad \qquad & \\[-0.1cm]
\theta_p(x_p) - \theta_p(x_p^\ast) + (x_p-x_p^\ast)^T(- A_p^T\lambda^\ast) \ge 0,  & \forall\; x_p\in {\cal X}_p, \\[0.2cm]
(\lambda-\lambda^\ast)^T(\sum_{i=1}^pA_ix_i^\ast-b)\geq 0,  &  \forall \; \lambda\in \Re^m.
\end{array} \right.
\end{equation}
Furthermore,  all inequalities in \eqref{VI} can be rewritten as the following more compact form:
\begin{subequations}\label{VI1}
\begin{equation}\label{OVI}
  \hbox{VI}(\Omega,F,\theta): \quad  w^*\in \Omega, \quad \theta(x) -\theta(x^*) + (w-w^*)^T F(w^*) \ge 0, \quad  \forall  \; w\in\Omega,
\end{equation}
with
\begin{eqnarray}\label{VI-S}
 \theta=\sum_{i=1}^p\theta_i, & & \Omega=\mathcal{X}_1\times\cdots\times\mathcal{X}_p\times\Re^m,  \nonumber\\
 x= \left(\!\begin{array}{c}
                     x_1\\[-0.1cm]
                     \vdots\\[-0.1cm]
                     x_p
 \end{array} \!\right),  & &
 w= \left(\!\begin{array}{c}
                     x_1\\[-0.1cm]
                     \vdots\\[-0.1cm]
                     x_p\\
                     \lambda
 \end{array}\! \right) \quad \hbox{and} \quad
F(w)=\left(\!\begin{array}{c}
           -\mathcal{A}^T\lambda \\[0.1cm]
           \mathcal{A}x-b \\
            \end{array}\!\right)=\left(\!\begin{array}{c}
           -A_1^T\lambda \\
                     \vdots       \\
           -A_p^T\lambda \\[0.1cm]
           \sum_{i=1}^p A_ix_i-b \\
            \end{array}\!\right).
\end{eqnarray}
\end{subequations}
Note that the operator $F$ defined in \eqref{VI-S} is monotone, because that
\begin{equation}\label{Skew-S}
  (w-\bar{w})^T(F(w)-F(\bar{w}))\equiv0, \quad \forall\; w,\;\bar{w}\in\Re^{n+m}.
\end{equation}
Also, the solution set of \eqref{VI1} is the set of saddle points of the Lagrangian function \eqref{lagra} of the model \eqref{A-Problem-M}.

\subsection{Some matrices}
To simplify the presentation of analysis,  we first define the following four matrices:
\begin{equation}\label{Matrix-wQ}
Q = \left[\!\!\begin{array}{ccccc}
       \beta A_1^TA_1 &                             &              &                              &  A_1^T       \\[0.1cm]
                                  & \beta A_2^TA_2  &              &                              & A_2^T        \\[0.1cm]
                                  &                             &   \ddots &                              &    \vdots       \\[0.1cm]
                                  &                             &              &  \beta A_p^TA_p  &   A_p^T      \\[0.1cm]
                      -  A_1  &   -A_2                  &  \cdots  &  -A_p                    &  \frac{1}{\beta}I_m
 \end{array}\!\!\right], \quad
 P =\left[\!\!\begin{array}{ccccc}
A_1   &           &                &           &    0          \\[0.1cm]
          &  A_2  &                &           &    0          \\[0.1cm]
          &           &   \ddots   &           &  \vdots     \\[0.1cm]
          &           &                &  A_p  &    0          \\[0.1cm]
0        &    0     &   \cdots   &     0    &   I_m
 \end{array}\!\!\right],
\end{equation}
\begin{equation}\label{Matrix-QS}
\mathcal{D}=\left[\!\!\begin{array}{ccccc}
\beta I_m   &                   &                 &                   &    0          \\[0.1cm]
                  &  \beta I_m  &                 &                   &    0          \\[0.1cm]
                  &                   &   \ddots    &                   &  \vdots     \\[0.1cm]
                  &                   &                 &  \beta I_m  &    0          \\[0.1cm]
0                &        0         &   \cdots    &          0       &   \frac{1}{\beta}I_m
 \end{array}\!\!\right]  \quad \hbox{and} \quad
 \mathcal{Q} = \left[\!\!\begin{array}{ccccc}
     \beta I_m  &              &             &             &   I_m                   \\[0.1cm]
             &     \beta I_m   &             &             &   I_m                   \\[0.1cm]
             &              &  \ddots &              &  \vdots                \\[0.1cm]
             &              &             &    \beta I_m   &   I_m                  \\[0.1cm]
  - I_m  &   - I_m   &  \cdots &   -I_m   &   \frac{1}{\beta}I_m
     \end{array}\!\!\right],
\end{equation}
where $I_m$ is the $m\times m$ identity matrix. It is obvious that the matrices $Q$, $P$, $\mathcal{D}$ and $\mathcal{Q}$ are all ${(p+1)\times (p+1)}$ partitioned, and they satisfy the condition
\begin{equation}\label{PQP}
  Q=P^T\mathcal{Q}P \quad \hbox{and} \quad \mathcal{Q}^T+\mathcal{Q}=2\mathcal{D}.
\end{equation}

\subsection{Some properties}
With the matrices $\mathcal{D}$ and $\mathcal{Q}$ defined in \eqref{Matrix-QS},  we further define two matrices $\mathcal{M}$ and $\mathcal{H}$ as
\begin{equation}\label{Matrix-M}
  \mathcal{M} = \mathcal{Q}^{-T}\mathcal{D} \quad \hbox{and} \quad \mathcal{H}=\mathcal{Q}\mathcal{D}^{-1}\mathcal{Q}^T.
\end{equation}
Below we show that the matrices $\mathcal{M}$ and $\mathcal{H}$ in \eqref{Matrix-M} have the following  properties.

\begin{property}
For the matrix $\mathcal{M}$ defined in \eqref{Matrix-M}, it holds that
\end{property}
\begin{equation}\label{Matrix-M-T}
  \mathcal{M} =
     \left[\begin{array}{ccccc}
            I_m &      0  &  \cdots   &   0  & 0   \\
              0     &    \ddots  & \ddots     & \vdots    & \vdots \\
             \vdots   & \ddots  &   \ddots  &   0 &   0\\
              0  &  \cdots    &   0 &     I_m &        0\\
             0  &     0   &  \cdots & 0  &  I_m
                               \end{array}\!\!\right]    - \frac{1}{p+1}
\left[\begin{array}{ccccc}
 I_m          &   I_m         &  \cdots   &  I_m          &  -\frac{1}{\beta}I_m   \\
 I_m          &   I_m         &  \cdots   &  I_m          &  -\frac{1}{\beta}I_m \\
 \vdots       &   \vdots      &  \ddots  &   \vdots      &  \vdots   \\
 I_m          &   I_m         &  \cdots   &  I_m          &  -\frac{1}{\beta}I_m\\
 \beta I_m &   \beta I_m &  \cdots  &  \beta I_m  &  p I_m
                               \end{array}\right].
\end{equation}
\begin{proof}
To begin with, let us define two more $(p+1)\times (p+1)$ matrices $\mathcal{Q}_0$ and $\mathcal{D}_0$:
\begin{equation}\label{Q0}
  \mathcal{Q}_0 =
\left[\begin{array}{cccc}
\beta  &               &            &   1       \\
          &   \ddots  &            &  \vdots \\
          &               &  \beta  &   1       \\
 -1      &  \cdots   &     -1    &   \frac{1}{\beta}
\end{array}\right] =
\left[\begin{array}{cc}
  \beta I_p &  e_p  \\[0.1cm]
 -e_p^T   &   \frac{1}{\beta}
 \end{array}\right]\quad \hbox{and} \quad  \mathcal{D}_0 =
 \left[\begin{array}{cccc}
  \beta   &              &           &   0       \\
            &  \ddots   &           &  \vdots \\
            &               &  \beta &   0       \\
      0    &  \cdots   &     0    &   \frac{1}{\beta}
\end{array}\right] =
\left[\begin{array}{cc}
  \beta I_p &  0  \\[0.1cm]
   0    &   \frac{1}{\beta}
 \end{array}\right],
\end{equation}
where $e_p\in \Re^p$ is a column vector whose elements are all $1$. Then, we have
\begin{eqnarray}\label{QT0}
 \mathcal{Q}_0^T  & = &
\left[\begin{array}{cc}
\beta I_p & - e_p  \\[0.1cm]
 e_p^T    &   \frac{1}{\beta}
\end{array}\right] =
\left[\begin{array}{cc}
  \beta I_p  &  0      \\[0.1cm]
  0     &  \frac{1}{\beta}
\end{array}\right] +
\left[\begin{array}{cc}
e_p    &    0    \\[0.1cm]
0        &    1
\end{array}\right]_{(p+1)\times 2}
\left[\begin{array}{cc}
0           &    -1    \\[0.1cm]
e_p^T   &     0
\end{array}\right]_{2\times(p+1)}  \nonumber \\[0.2cm]
&=&
\left[\begin{array}{cc}
  \beta I_p  &  0      \\[0.1cm]
  0     &  \frac{1}{\beta}
\end{array}\right]\Bigg(\left[\begin{array}{cc}
  I_p  &  0      \\[0.1cm]
  0     &  1
\end{array}\right] +
\left[\begin{array}{cc}
\frac{1}{\sqrt{\beta}}e_p    &    0    \\[0.1cm]
0        &    \sqrt{\beta}
\end{array}\right]
\left[\begin{array}{cc}
0     &    -\frac{1}{\sqrt{\beta}}    \\[0.1cm]
\sqrt{\beta}e_p^T   &     0
\end{array}\right]\Bigg).
\end{eqnarray}
Setting
$$ U =
\left[\begin{array}{cc}
\frac{1}{\sqrt{\beta}}e_p   &  0 \\[0.1cm]
 0     &  \sqrt{\beta}
\end{array}\right]  \quad \hbox{and} \quad V =
\left[\begin{array}{cc}
0     &   \sqrt{\beta}e_p   \\[0.1cm]
-\frac{1}{\sqrt{\beta}}    &    0
\end{array}\right],$$
and using the Sherman-Morrison-Woodbury formula, we obtain that
\begin{eqnarray}\label{M0-T}
\mathcal{M}_0  & := & \mathcal{Q}_0^{-T}\mathcal{D}_0 \nonumber \\
& = & [\mathcal{D}_0(I_{p+1} + UV^T)]^{-1}\mathcal{D}_0      \nonumber \\
& = & [I_{p+1} + UV^T)]^{-1} \nonumber \\
& = & I_{p+1} - U(I_2 + V^T U)^{-1} V^T     \nonumber \\
& = & I_{p+1}-
\left[\begin{array}{cc}
\frac{1}{\sqrt{\beta}}e_p  &   0 \\[0.1cm]
 0    &   \sqrt{\beta}
\end{array}\right]
\left[
\begin{array}{cc}
1 & -1 \\
p & 1 \\
\end{array}
\right]^{-1}
\left[\begin{array}{cc}
0           &  -\frac{1}{\sqrt{\beta}}   \\[0.1cm]
\sqrt{\beta}e_p^T   &   0
\end{array}\right]          \nonumber       \\
& = & I_{p+1} - \frac{1}{p+1}
\left[\begin{array}{cc}
\frac{1}{\sqrt{\beta}}e_p  &   0 \\[0.1cm]
0      &  \sqrt{\beta}
\end{array}\right]
\left[\begin{array}{cc}
1   &  1 \\
-p  &  1
\end{array}\right]
\left[\begin{array}{cc}
0          &  -\frac{1}{\sqrt{\beta}} \\[0.1cm]
\sqrt{\beta}e_p^T  &   0
\end{array}\right]     \nonumber  \\
& = &  I_{p+1} - \frac{1}{p+1}
\left[\begin{array}{cc}
e_p e_p^T   &   -\frac{1}{\beta}e_p \\[0.1cm]
\beta e_p^T          &     p
\end{array}\right].
\end{eqnarray}
Note that $\mathcal{Q}=\mathcal{Q}_0\otimes I_m$ and $\mathcal{D}=\mathcal{D}_0\otimes I_m$, where $\otimes$ denotes the Kronecker product of two matrices. Then, it follows from basic properties of the Kronecker product that
$$\mathcal{M}=\mathcal{Q}^{-T}\mathcal{D}=(\mathcal{Q}_0\otimes I_m)^{-T}(\mathcal{D}_0\otimes I_m)=(\mathcal{Q}_0^{-T}\otimes I_m^{-T})(\mathcal{D}_0\otimes I_m)= (\mathcal{Q}_0^{-T}\mathcal{D}_0)\otimes (I_m^{-T}I_m)=\mathcal{M}_0\otimes I_m,$$
which is indeed the matrix defined in \eqref{Matrix-M-T}. The proof is complete.
\end{proof}
\begin{property}
For the matrix $\mathcal{H}$ defined in \eqref{Matrix-M}, it holds that
\begin{equation}\label{Matrix-H-T}
\mathcal{H} =
       \left[\begin{array}{ccccc}
       2\beta I_m &    \beta I_m  &  \cdots   &  \beta I_m  & 0   \\
       \beta I_m    &    \ddots  & \ddots     & \vdots    & \vdots \\
       \vdots   & \ddots  &   \ddots  &   \beta I_m &   0\\
       \beta I_m  &  \cdots    &   \beta I_m &    2\beta I_m &        0\\
       0  &     0   &  \cdots & 0  & \frac{1}{\beta}(1+p) I_m
       \end{array}\!\!\right].
\end{equation}
\end{property}
\begin{proof}
Recall \eqref{Q0} and obtain
$$ \mathcal{H}_0:=\mathcal{Q}_0\mathcal{D}_0^{-1}\mathcal{Q}_0^T= \left[\begin{array}{cc}
  \beta I_p &  e_p  \\[0.1cm]
 -e_p^T   &   \frac{1}{\beta}
 \end{array}\right]
\left[\begin{array}{cc}
  \frac{1}{\beta} I_p &  0  \\[0.1cm]
   0    &   \beta
 \end{array}\right]
 \left[\begin{array}{cc}
\beta I_p & - e_p  \\[0.1cm]
 e_p^T    &   \frac{1}{\beta}
\end{array}\right] =
\left[\begin{array}{cc}
 \beta (I_p+e_pe_p^T)  &  0      \\[0.1cm]
  0     &  \frac{1}{\beta}(1+p)
\end{array}\right]. $$
Then, it follows from basic properties of the Kronecker product that
$$\mathcal{H} = \mathcal{Q}\mathcal{D}^{-1}\mathcal{Q}^T=(\mathcal{Q}_0\otimes I_m)(\mathcal{D}_0^{-1}\otimes I_m^{-1})(\mathcal{Q}_0^T\otimes I_m)=(\mathcal{Q}_0\mathcal{D}_0^{-1}\mathcal{Q}_0^T)\otimes (I_mI_m^{-1}I_m)=\mathcal{H}_0\otimes I_m.$$
This is indeed \eqref{Matrix-H-T}, and the proof is complete.
\end{proof}

\subsection{Notations}

It is clear that $\mathcal{M}_0$ defined in \eqref{M0-T} is the sum of an identity matrix and a rank-two matrix,  and that $\mathcal{M}=\mathcal{M}_0\otimes I_m$. The matrix $\mathcal{M}$ is thus a \textbf{block rank-two matrix}.  At the end of this section, we also define some notations for further analysis. More specifically, we denote by $\Omega^\ast$  the solution set of the VI \eqref{VI1} and define $\xi\in\Re^{(p+1)m\times (p+1)m}$ as
\begin{equation}\label{Xi-notation}
  \xi:= Pw
\end{equation}
for any $w\in\Omega$, where $P$ is the matrix defined in \eqref{Matrix-wQ}.   Accordingly, we further define
\begin{equation}\label{Xi-notation1}
\Xi=\big\{Pw \mid w\in  \Omega\big\} \quad \hbox{and} \quad  \Xi^\ast=\big\{ Pw^\ast  \mid  w^*\in  \Omega^*\big\}.
\end{equation}

Note that
\begin{equation}
\xi= Pw=\left[\!\!\begin{array}{ccccc}
A_1   &           &                &           &    0          \\[0.1cm]
          &  A_2  &                &           &    0          \\[0.1cm]
          &           &   \ddots   &           &  \vdots     \\[0.1cm]
          &           &                &  A_p  &    0          \\[0.1cm]
0        &    0     &   \cdots   &     0    &   I_m
 \end{array}\!\!\right]\left(\!\begin{array}{c}
                     x_1\\[-0.1cm]
                     \vdots\\[-0.1cm]
                     x_p\\
                     \lambda
 \end{array}\! \right)=\left(\!\begin{array}{c}
                  A_1    x_1\\[-0.1cm]
                     \vdots\\[-0.1cm]
                  A_p    x_p\\
                     \lambda
 \end{array}\! \right),
\end{equation}
and recall the matrix $\mathcal{H}$ defined in \eqref{Matrix-H-T}. For convenience, we also define the following $p$-partitioned notations:
\begin{equation}\label{Matrix-H-T2}
\hat{\xi}=\left(\!\begin{array}{c}
                  A_1    x_1\\[-0.1cm]
                     \vdots\\[-0.1cm]
                  A_p    x_p
 \end{array}\! \right)  \quad \hbox{and} \quad
\hat{\mathcal{H}} =
\left[\begin{array}{cccc}
       2\beta I_m &    \beta I_m  &  \cdots   &  \beta I_m     \\
       \beta I_m    &    \ddots  & \ddots     & \vdots     \\
       \vdots   & \ddots  &   \ddots  &   \beta I_m    \\
       \beta I_m  &  \cdots    &   \beta I_m &   2\beta I_m \\
       \end{array}\!\!\right].
\end{equation}

\section{Algorithm and remarks}\label{sec3}
\setcounter{equation}{0}
\setcounter{remark}{0}
\setcounter{proposition}{0}

In this section, we present a rank-two relaxed parallel splitting version of the ALM with a step size in $(0,2)$ for the separable convex programming problem \eqref{A-Problem-M} with $p\ge 2$. We also elaborate on its difference from the algorithms in \cite{Deng2017,HeHouYuanSIOPT,HeTaoYuanIMA,HXY2016}.

\subsection{Algorithm}

Recall that the matrix $\mathcal{M}$ defined in \eqref{Matrix-M} (specified by \eqref{Matrix-M-T}) is block rank-two, and that we use the notations  $\xi^k=(A_1x_1^k;\ldots;A_p x_p^k; \lambda^k)$ and  $\tilde{\xi}^k=(A_1\tilde{x}_1^k;\ldots;A_p\tilde{x}_p^k;\tilde{\lambda}^k)$.
To solve the separable convex programming problem \eqref{A-Problem-M} with $p\ge 2$, with given $\xi^k$, the rank-two relaxed parallel splitting version of the ALM generates the new iterate $\xi^{k+1}$ via
 \begin{subequations} \label{Rank2M}
\begin{numcases}{}
\label{Pall-XL}\left\{
    \begin{array}{ll}
    \tilde{x}_1^{k} \; \in \; \arg\min\big\{  L(x_1,x_2^k, \ldots, x_p^k, \lambda^k)  +\frac{\beta}{2} \|A_1x_1-A_1x_1^k\|^2  \mid x_1\in\mathcal{X}_1  \big\},\\
       \quad \qquad \qquad \vdots \\
   \tilde{x}_i^{k}  \; \in \; \arg\min\big\{ L(x_1^k,\ldots,x_{i-1}^k, x_i, x_{i+1}^k, \ldots, x_p^k, \lambda^k)   +\frac{\beta}{2} \|A_ix_i-A_ix_i^k\|^2   \mid   x_i \in\mathcal{X}_i \big\},\\
       \quad \qquad  \qquad \vdots \\
   \tilde{x}_p^{k} \; \in \; \arg\min \big\{ L(x_1^k,\ldots, x_{p-1}^k, x_p, \lambda^k)  +\frac{\beta}{2} \|A_px_p-A_px_p^k\|^2  \mid   x_p \in\mathcal{X}_p \big\},\\[0.2cm]
   \tilde{\lambda}^{k} \; = \; \lambda^k - \beta(\sum_{i=1}^p A_i x_i^k-b),
    \end{array}
      \right. \\[0.2cm]
\label{CorrectionS} \;\; \xi^{k+1} \,=\; \xi^k-\alpha \mathcal{M}(\xi^k-\tilde{\xi}^k)\; \;\hbox{with} \;\; \alpha \in (0,2).
\end{numcases}
\end{subequations}
For obvious reasons, we call \eqref{Pall-XL} and (\ref{CorrectionS}) the parallel splitting ALM step and the rank-two relaxation step, respectively.

\begin{remark}
For the parallel splitting ALM step \eqref{Pall-XL}, ignoring some constant terms, we can simplify its $x_i$-subproblems as
\begin{equation}\label{Core-x1}
  \tilde{x}_i^k \in   \arg\min\Big\{\theta_i(x_i)+\frac{\beta}{2}\big\|A_ix_i-\big(A_ix_i^k+\frac{1}{\beta}\lambda^k\big)\big\|^2  \mid x_i \in\mathcal{X}_i \Big\}.
\end{equation}
Similarly, the $x_i$-subproblem in \eqref{PSALM-P} can be rewritten as
\begin{equation}\label{Core-x2}
  x_i^{k+1} \in \arg\min\Big\{\theta_i(x_i)+\frac{\beta}{2}\big\|A_ix_i-\big(-\sum_{j\neq i}A_jx_j^k+b+\frac{1}{\beta}\lambda^k\big)\big\|^2  \mid x_i \in\mathcal{X}_i \Big\},
\end{equation}
and the $x_i$-subproblems in \eqref{PJALMP} can be rewritten as
\begin{equation}\label{Core-x3}
  x_i^{k+1}  \in \arg\min\Big\{\theta_i(x_i)+\frac{(1+\tau)\beta}{2}\big\|A_ix_i-q^k\big\|^2  \mid x_i \in\mathcal{X}_i \Big\},
\end{equation}
with
$$q^k=\frac{1}{1+\tau}(-\sum_{j\neq i}A_jx_j^k+b+\frac{1}{\beta}\lambda^k)+\frac{1}{1+\tau}\tau A_ix_i^k.$$
Therefore, the $x_i$-subproblems in \eqref{Pall-XL} are analytically of the same difficulty as those in \eqref{DADMM}, \eqref{PSALM-P} and \eqref{PJALMP}, only with the difference in the constant vectors and the coefficients of the respective quadratic terms. Recall that the additional parameter $\tau$ subject to the condition (\ref{PJALMC}) is removed in (\ref{Pall-XL}).
\end{remark}

\begin{remark}
For the rank-two relaxation step (\ref{CorrectionS}), it follows from the definition of $\mathcal{M}$ in \eqref{Matrix-M-T} that it can be specified as
\begin{equation}\label{Correction}
  \left(\!\!\!\begin{array}{c}
                   A_1  x_1^{k+1}\\
                    \vdots \\
                   A_p   x_p^{k+1}\\[0.1cm]
                   \lambda^{k+1}  \end{array} \!\!\!\right)=
\left(\!\!\!\begin{array}{c}
                   A_1  x_1^k\\
                    \vdots \\
                   A_p   x_p^k \\[0.1cm]
                   \lambda^{k}  \end{array}\! \!\!\right)-\alpha
\left(\!\!\!\begin{array}{c}
                        A_1  x_1^{k}- A_1 \tilde{x}_1^{k} \\
                        \vdots \\
                        A_p   x_p^{k}- A_p\tilde{x}_p^{k}\\[0.1cm]
                        \lambda^{k} -\tilde{\lambda}^k  \end{array} \!\!\!\right) + \frac{\alpha}{p+1}
   \left(\!\!\! \begin{array}{c}
                         \sum_{i=1}^p(A_ix_i^k -A_i\tilde{x}_i^k) - \frac{1}{\beta} (\lambda^{k} -\tilde{\lambda}^k) \\
                         \vdots \\
                         \sum_{i=1}^p(A_ix_i^k -A_i\tilde{x}_i^k) - \frac{1}{\beta} (\lambda^{k} -\tilde{\lambda}^k)   \\[0.1cm]
                         \beta\sum_{i=1}^p(A_ix_i^k -A_i\tilde{x}_i^k)  +p(\lambda^{k} -\tilde{\lambda}^k)  \end{array}\!\!\! \right).
\end{equation}
It is clear that only $A_ix_i^k$ ($i=1,\ldots,p$), rather than $x_i^k$ ($i=1,\ldots,p$), and $\lambda^k$ are needed as the input for the $(k+1)$-th iteration to solve the $x_i$-subproblems in (\ref{Pall-XL}). Hence, the rank-two relaxation step (\ref{Correction}) can be executed in terms of  $\{A_ix_i \}(i=1,\ldots,p) $ and $\{\lambda\}$, which is extremely easy, while $x_i$'s only need to be solved once at the last iteration.
\end{remark}

\subsection{Connection with other algorithms}

Note that the proposed algorithm (\ref{Rank2M}), as well as the existing algorithms \eqref{PSALM-P}-\eqref{PSALM-C} and \eqref{PJALMP}, are all based on the fact that the direct parallel splitting version of the ALM (\ref{PSALM-P}) is not necessarily convergent, and their common goal is modifying the root scheme (\ref{PSALM-P}) slightly to maintain the subproblems in (\ref{PSALM-P}) as much as possible while the convergence can be guaranteed. Thus, it is meaningful to discern the difference of these algorithms by calibrating their respective difference from the root scheme (\ref{PSALM-P}). For this purpose, let us denote
\[\label{M1}
\mathcal{M}_1 =
  \left[\!\!\begin{array}{ccccc}
  I_{n_1}   &                   &                 &                   &    0          \\[0.1cm]
                  &    I_{n_2}  &                 &                   &    0          \\[0.1cm]
                  &                   &   \ddots    &                   &  \vdots     \\[0.1cm]
                  &                   &                 &   I_{n_p}  &    0          \\[0.1cm]
  -\beta A_1        &    -\beta A_2      &   \cdots     &     -\beta A_p     &  I_m
 \end{array}\!\!\right],
 \]
and recall the notation $w=(x_1;\ldots;x_p;\lambda)$ in \eqref{VI-S} as well as  \eqref{Core-x2}. Then, it is easy to see that the direct parallel splitting version of the ALM (\ref{PSALM-P}) can be represented purposively as the following prediction-correction framework.

\begin{center}\fbox{
 \begin{minipage}{16.3cm}
\smallskip
\noindent{\bf{Prediction-correction representation for the direct parallel splitting ALM (\ref{PSALM-P}). }}

\medskip
\noindent{\textbf{(Prediction Step)}} With given $w^k=(x_1^k; \ldots; x_p^k; \lambda^k)$, $\tilde{w}^k=(\tilde{x}_1^k;\ldots; \tilde{x}_p^k;\tilde{\lambda}^k)$ satisfies
\begin{subequations}\label{PALM-PC}
\begin{equation}\label{PALM-Pre}
\left\{
    \begin{array}{lcl}
    \tilde{\lambda}^{k} & = & \lambda^k - \beta(\sum_{i=1}^p A_i x_i^k-b),\\[0.2cm]
    \tilde{x}_1^{k} & = & \arg\min\big\{\theta_1(x_1) + \frac{\beta}{2} \|A_1x_1-(A_1x_1^k+\frac{1}{\beta}\tilde{\lambda}^k)\|^2  \mid x_1\in\mathcal{X}_1  \big\},\\[-0.1cm]
       & & \qquad \vdots \\[-0.1cm]
   \tilde{x}_i^{k}  & = & \arg\min\big\{\theta_i(x_i)  + \frac{\beta}{2} \|A_ix_i-(A_ix_i^k+\frac{1}{\beta}\tilde{\lambda}^k)\|^2   \mid   x_i \in\mathcal{X}_i \big\},\\[-0.1cm]
       & &  \qquad \vdots \\[-0.1cm]
   \tilde{x}_p^{k} & = & \arg\min \big\{\theta_p(x_p) + \frac{\beta}{2} \|A_px_p-(A_px_p^k+\frac{1}{\beta}\tilde{\lambda}^k)\|^2  \mid   x_p \in\mathcal{X}_p \big\}.
    \end{array}
  \right.
\end{equation}
\noindent{\textbf{(Correction Step)}}  With the predictor $\tilde{w}^k$ represented by \eqref{PALM-Pre}, the new iterate $w^{k+1}$  can be generated by
\begin{equation}\label{PALM-Cor}
w^{k+1}=w^k-\mathcal{M}_1(w^k-\tilde{w}^k).
\end{equation}
\end{subequations}
 \end{minipage}}
\end{center}

Now, with the prediction-correction representation \eqref{PALM-Pre}-\eqref{PALM-Cor} of the root scheme \eqref{PSALM-P}, we calibrate the difference of various algorithms from the benchmark \eqref{PSALM-P} by representing them also in the prediction-correction framework.

\begin{itemize}
  \item For the modified Jacobian splitting ALM \eqref{PSALM-P}-\eqref{PSALM-C} proposed in \cite{HeHouYuanSIOPT} (denoted by ``JSALM" for short), it keeps \eqref{PALM-Pre} but with the more conservative  correction step
 \begin{equation}\label{JSALM-C}
 w^{k+1}=w^k-\alpha \mathcal{M}_1(w^k-\tilde{w}^k),
 \end{equation}
where $\alpha$  is required to satisfy the condition \eqref{PSALM-step}.
By comparing \eqref{PALM-Cor} and \eqref{JSALM-C}, it can be intuitively understood that the JSALM \eqref{PSALM-P}-\eqref{PSALM-C} overcomes the divergence of the direct parallel splitting version of the ALM (\ref{PSALM-P}) by replacing the correction step \eqref{PALM-Cor} with the more conservative one \eqref{JSALM-C}. This strategy ensures the convergence theoretically, but it becomes more conservative because the step size $\alpha$ in \eqref{JSALM-C} is diminishing when $p$ increases.

\item For the proximal Jacobian splitting ALM \eqref{PJALMP} proposed in \cite{Deng2017,HeTaoYuanIMA,HXY2016} (denoted by ``PJALM" for short),  note that the only difference between \eqref{PSALM-P} and \eqref{PJALMP} is the additional proximal terms regarding the $x_i$-subproblems. The PJALM \eqref{PJALMP} thus can be represented as a prediction-correction framework, with the same correction step \eqref{PALM-Cor} while its prediction step is
    \begin{equation}\label{PJALM-P}
    \left\{
    \begin{array}{lcl}
    \tilde{\lambda}^{k} & = & \lambda^k - \beta(\sum_{i=1}^p A_i x_i^k-b),\\[0.2cm]
    \tilde{x}_1^{k} & = & \arg\min\big\{\theta_1(x_1) + \frac{\beta}{2} \|A_1x_1-(A_1x_1^k+\frac{1}{\beta}\tilde{\lambda}^k)\|^2 +\frac{\tau\beta}{2}\|A_1x_1-A_1x_1^k\|^2 \mid x_1\in\mathcal{X}_1  \big\},\\[-0.1cm]
       & & \qquad \vdots \\[-0.1cm]
    \tilde{x}_i^{k}  & = & \arg\min\big\{\theta_i(x_i)  + \frac{\beta}{2} \|A_ix_i-(A_ix_i^k+\frac{1}{\beta}\tilde{\lambda}^k)\|^2  +\frac{\tau\beta}{2}\|A_ix_i-A_ix_i^k\|^2 \mid   x_i \in\mathcal{X}_i \big\},\\[-0.1cm]
       & &  \qquad \vdots \\[-0.1cm]
    \tilde{x}_p^{k} & = & \arg\min \big\{\theta_p(x_p) + \frac{\beta}{2} \|A_px_p-(A_px_p^k+\frac{1}{\beta}\tilde{\lambda}^k)\|^2+\frac{\tau\beta}{2}\|A_px_p-A_px_p^k\|^2  \mid   x_p \in\mathcal{X}_p \big\}.
    \end{array}
    \right.
   \end{equation}
  Comparing with \eqref{PALM-Pre} and \eqref{PJALM-P}, we know that the PJALM  \eqref{PJALMP} adjusts the prediction step \eqref{PALM-Pre} by proximally regularizing all the $x_i$-subproblems in \eqref{PALM-Pre} with the proximal coefficient $\frac{\tau\beta}{2}$ and $\tau$ is required to satisfy the condition \eqref{PJALMC}. Hence, the PJALM  \eqref{PJALMP} overcomes the divergence of the direct parallel splitting version of the ALM \eqref{PSALM-P} by replacing the prediction step (\ref{PALM-Pre}) with the more conservative one \eqref{PJALM-P}. Note that larger values of $p$ imply larger values of $\tau$ and hence smaller step sizes for solving the $x_i$-subproblems in \eqref{PJALM-P}.

  \item For the rank-two relaxed parallel splitting version of the ALM \eqref{Rank2M}, according to \eqref{Core-x1}, we can rewrite the parallel splitting ALM step \eqref{Pall-XL} as
\begin{equation}\label{Pall-XL-1}
  \left\{
    \begin{array}{lcl}
   \tilde{\lambda}^{k} & = & \lambda^k - \beta(\sum_{i=1}^p A_i x_i^k-b),  \\[0.3cm]
    \tilde{x}_1^{k} & \in & \arg\min\big\{\theta_1(x_1) + \frac{\beta}{2} \|A_1x_1-(A_1x_1^k+\frac{1}{\beta}\lambda^k)\|^2  \mid x_1\in\mathcal{X}_1  \big\},\\
       & & \qquad \vdots \\
   \tilde{x}_i^{k}  & \in & \arg\min\big\{\theta_i(x_i)  + \frac{\beta}{2} \|A_ix_i-(A_ix_i^k+\frac{1}{\beta}\lambda^k)\|^2   \mid   x_i \in\mathcal{X}_i \big\},\\
       & &  \qquad \vdots \\
   \tilde{x}_p^{k} & \in & \arg\min \big\{\theta_p(x_p)  + \frac{\beta}{2} \|A_px_p-(A_px_p^k+\frac{1}{\beta}\lambda^k)\|^2  \mid   x_p \in\mathcal{X}_p \big\}.
    \end{array}
  \right.
\end{equation}
Hence, the rank-two relaxed parallel splitting version of the ALM \eqref{Rank2M} can also be represented as a prediction-correction framework, i.e., \eqref{Pall-XL-1}+\eqref{CorrectionS}. Compared with \eqref{PALM-Pre}, the only difference in \eqref{Pall-XL-1} is the constant vectors regarding $\lambda$ in the quadratic terms of the $x_i$-subproblems, while all major features and structures of \eqref{PALM-Pre} are maintained in \eqref{Pall-XL-1}. The $x_i$-subproblems in \eqref{Pall-XL-1} also differ from those in \eqref{PJALM-P} in that the coefficients of the quadratic terms are irrelevant with $\tau$ and thus do not depend on $p$. Recall that the block rank-two matrix $\mathcal{M}$ is used to correct the relaxation step \eqref{CorrectionS}. Hence, the proposed rank-two relaxed parallel splitting version of the ALM \eqref{Rank2M} adjusts both the prediction and correction steps of the direct parallel splitting version of the ALM \eqref{PSALM-P}, but more mildly than \eqref{JSALM-C} and \eqref{PJALM-P}. Both the new prediction step \eqref{Pall-XL-1} and the correction step \eqref{CorrectionS} do not generate any more difficult subproblems, nor do they require any new conditions on new parameters. Meanwhile, it maintains the advantage of a step size in $(0,2)$ for the relaxation step, while the convergence is still ensured rigorously.
\end{itemize}

\section{Convergence}\label{sec4}
\setcounter{equation}{0}
\setcounter{remark}{0}
\setcounter{proposition}{0}

In this section, we conduct convergence analysis for the proposed new algorithm (\ref{Rank2M}). As mentioned, to execute the parallel splitting ALM step \eqref{Pall-XL}, only $A_ix_i \, (i=1,\ldots,p)$ and $\lambda$ are required. Thus, the convergence analysis is conducted in the context of the sequence $\big\{\xi^k\big\}$. Recall that we also use the notation $\tilde{w}^k$ to denote the output of the parallel splitting ALM step \eqref{Pall-XL}. The following lemma characterizes the difference of $\tilde{w}^k$ from a solution point of the VI \eqref{VI1}.

\begin{lemma}\label{VI-Lemma}
Let $\xi$ be defined in \eqref{Xi-notation}, and $\tilde{w}^k=(\tilde{x}_1^k; \ldots; \tilde{x}_p^k; \tilde{\lambda}^k)$  be the output of the parallel splitting ALM step \eqref{Pall-XL} with given input $\xi^k=(A_1x_1^k;\ldots;A_px_p^k;\lambda^k)$.  Then we have
\begin{equation}\label{Prediction-w}
\tilde{w}^k\in \Omega,\quad  \theta(x) - \theta(\tilde{x}^k) +  (w- \tilde{w}^k )^T F(\tilde{w}^k) \ge (\xi - \tilde{\xi}^k )^T\mathcal{Q} (\xi^k-\tilde{\xi}^k), \quad  \forall \; w\in\Omega,
\end{equation}
where $\mathcal{Q}$ is the matrix defined in \eqref{Matrix-QS}.
\end{lemma}
\begin{proof}
For each $x_i$-subproblem in \eqref{Pall-XL}, it follows from Lemma \ref{CP-TF} that
$$\tilde{x}_i^k\in {\cal X}_i, \quad  \theta_i(x_i)-\theta_i(\tilde{x}_i^k) + (x_i- \tilde{x}_i^k)^T\bigl[- A_i^T\lambda^k  + \beta A_i^T(A_i\tilde{x}_i^k -A_ix_i^k) \bigr] \ge 0, \quad \forall \; x_i\in {\cal X}_i,$$
which can be further rewritten as
\begin{equation}\label{Full-px}
\tilde{x}_i^k\in {\cal X}_i, \quad  \theta_i(x_i)-\theta_i(\tilde{x}_i^k) +  (x_i- \tilde{x}_i^k)^T\bigl[-A_i^T\tilde{\lambda}^k + \beta A_i^T(A_i\tilde{x}_i^k -A_ix_i^k) + A_i^T(\tilde{\lambda}^k-\lambda^k) \bigr] \ge 0, \quad \forall \; x_i\in {\cal X}_i.
\end{equation}
For the $\lambda$-subproblem in  \eqref{Pall-XL}, we have
$$\Big(\sum_{i=1}^p A_i \tilde{x}_i^{k}-b\Big) -\sum_{i=1}^p (A_i\tilde{x}_i^{k}-A_ix_i^k)+\frac{1}{\beta}(\tilde{\lambda}^{k} - \lambda^k)=0,$$
which is also equivalent to
\begin{equation}\label{Full-pl}
  \tilde{\lambda}^k \in \Re^m, \quad   (\lambda-\tilde{\lambda}^k)^T\Big[(\sum_{i=1}^p A_i \tilde{x}_i^{k}-b) -\sum_{i=1}^p (A_i\tilde{x}_i^{k}-A_ix_i^k)+\frac{1}{\beta}(\tilde{\lambda}^{k} - \lambda^k)\Big]\geq0, \quad \forall \; \lambda\in \Re^m.
\end{equation}
Adding \eqref{Full-px} and \eqref{Full-pl}, we have
\begin{eqnarray*}
\lefteqn{\sum_{i=1}^p\theta_i(x_i)-\sum_{i=1}^p\theta_i(\tilde{x}_i^k)+
\left(\!\!\begin{array}{c}
     x_1-\tilde{x}_1^k\\
     \vdots\\
     x_p-\tilde{x}_p^k\\[0.1cm]
    \lambda-\tilde{\lambda}^k\\
\end{array}\!\!\right)^T
\left(\!\!\begin{array}{c}
     - A_1^T\tilde{\lambda}^k\\
     \vdots\\
     - A_p^T\tilde{\lambda}^k\\[0.1cm]
    \sum_{i=1}^p A_i \tilde{x}_i^{k}-b \\
\end{array}\!\!\right)} \\
& \geq &
\left(\!\!\begin{array}{c}
     A_1(x_1-\tilde{x}_1^k)\\
     \vdots\\
     A_p(x_p-\tilde{x}_p^k)\\[0.1cm]
    \lambda-\tilde{\lambda}^k\\
\end{array}\!\!\right)^T
\left(\!\!\begin{array}{c}
     \beta (A_1x_1^k -A_1\tilde{x}_1^k)  + (\lambda^k - \tilde{\lambda}^k) \\
     \vdots\\
     \beta (A_px_p^k -A_p\tilde{x}_p^k)  + (\lambda^k - \tilde{\lambda}^k) \\[0.1cm]
    -\sum_{i=1}^p (A_ix_i^k -A_i\tilde{x}_i^{k})+\frac{1}{\beta}( \lambda^k-\tilde{\lambda}^{k} )\\
\end{array}\!\!\right), \quad \forall\;
(x_1,\ldots,x_p,\lambda)\in\Omega.
\end{eqnarray*}
Recall the notations in \eqref{VI-S}, \eqref{Xi-notation}, \eqref{Xi-notation1}, and the matrix $\mathcal{Q}$ defined in \eqref{Matrix-QS}. The assertion of this lemma follows immediately.
\end{proof}

Recall the matrices $\mathcal{H}$ and $\mathcal{M}$ defined in \eqref{Matrix-M}. It holds that
\begin{equation}\label{Matrix-GQ}
  \mathcal{H}\mathcal{M}=\mathcal{Q} \quad \hbox{and} \quad \mathcal{M}^T\mathcal{H}\mathcal{M}=\mathcal{D}.
\end{equation}
The following theorem shows the contraction property of the sequence $\{\xi^k\}$.
\begin{theorem}\label{direction}
Let $\xi$ and $\Xi^\ast$ be defined in \eqref{Xi-notation} and \eqref{Xi-notation1}, respectively, and $\{\xi^{k}\}$ be the sequence generated by the proposed new algorithm \eqref{Rank2M} with $\alpha\in(0,2)$ for \eqref{A-Problem-M}, and $\tilde{w}^k$ be the output of the parallel splitting ALM step \eqref{Pall-XL}. Then, we have
\begin{equation}\label{HauptA1}
\theta(x) - \theta(\tilde{x}^k)  +  (w - \tilde{w}^k)^T F(w)\ge\frac{1}{2\alpha}\bigl( \|\xi- \xi^{k+1}\|_{\mathcal{H}}^2- \|\xi- \xi^{k}\|_{\mathcal{H}}^2\bigr)
       +  \frac{1}{2}(2-\alpha) \|{\xi}^k - \tilde{\xi}^k  \|_{\mathcal{D}}^2, \quad \forall \,  w \in {\Omega},
\end{equation}
and
\begin{equation}\label{XIcontr}
  \|\xi^{k+1} - \xi^*\|_{\mathcal{H}}^2 \le   \| \xi^k - \xi^*  \|_{\mathcal{H}}^2   - \alpha(2-\alpha)  \|\xi^k - \tilde{\xi}^k\|_\mathcal{D}^2, \quad \forall \;  \xi^*\in \Xi^\ast,
\end{equation}
where $\mathcal{D}$ and $\mathcal{H}$ are the matrices defined in \eqref{Matrix-QS} and \eqref{Matrix-M}, respectively.
\end{theorem}
\begin{proof}
It follows from \eqref{Skew-S} that \eqref{Prediction-w} is equivalent to
\begin{equation}\label{M-PRE-1}
  \theta(x) - \theta(\tilde{x}^k)  +  (w - \tilde{w}^k)^T F(w) \overset{\eqref{Skew-S}}{\equiv} \theta(x) - \theta(\tilde{x}^k)  +  (w - \tilde{w}^k)^T F(\tilde{w}^k) \ge
      (\xi  -\tilde{\xi}^k)^T\mathcal{Q}(\xi^k-\tilde{\xi}^k).
\end{equation}
Let us first refine the right-hand side of \eqref{M-PRE-1} as
\[ \label{LEM-M-a}  (\xi-\tilde{\xi}^{k})^T\mathcal{Q}(\xi^k  -\tilde{\xi}^{k})=   \frac{1}{2\alpha}\bigl( \|\xi- \xi^{k+1}\|_\mathcal{H}^2- \|\xi- \xi^{k}\|_\mathcal{H}^2\bigr)
       +  \frac{1}{2}(2-\alpha)  \|{\xi}^k - \tilde{\xi}^k  \|_\mathcal{D}^2.  \]
To this end, according to \eqref{Matrix-GQ}  and  \eqref{CorrectionS}, we have
\begin{equation}\label{LEM-M-b}
(\xi-\tilde{\xi}^{k})^T \mathcal{Q}(\xi^k  -\tilde{\xi}^{k}) \overset{\eqref{Matrix-GQ}}{=} (\xi-\tilde{\xi}^{k})^T\mathcal{H}\mathcal{M}(\xi^k  -\tilde{\xi}^{k}) \overset{\eqref{CorrectionS}}{=}  \frac{1}{\alpha}(\xi -\tilde{\xi}^{k})^T\mathcal{H}(\xi^k  -{\xi}^{k+1}).
\end{equation}
Using the identity
$$(a-b)^T\mathcal{H}(c-d) = \frac{1}{2} \big\{\|a-d\|_\mathcal{H}^2 -\|a-c\|_\mathcal{H}^2 \big\} + \frac{1}{2} \big\{\|c-b\|_\mathcal{H}^2 -\|d-b\|_\mathcal{H}^2 \big\}$$
with $ a=\xi$, $b=\tilde{\xi}^k$, $c=\xi^k$ and $d=\xi^{k+1}$, we get
\begin{equation}\label{LEM-M-c}
  (\xi-\tilde{\xi}^k)^T\mathcal{H}(\xi^k-\xi^{k+1}) = \frac{1}{2} \big\{ \|\xi-\xi^{k+1}\|_\mathcal{H}^2  -  \|\xi-\xi^k \|_\mathcal{H}^2 \big\} + \frac{1}{2}\big\{\|\xi^k -\tilde{\xi}^k\|_\mathcal{H}^2  -  \|\xi^{k+1}-\tilde{\xi}^k\|_\mathcal{H}^2\big\}.
\end{equation}
For the second term of the right-hand side of \eqref{LEM-M-c}, we have
\begin{eqnarray}  \label{LEM-M-d}
  \frac{1}{2}\big\{ \|{\xi}^k -\tilde{\xi}^k \|_\mathcal{H}^2  -   \|{\xi}^{k+1} -\tilde{\xi}^k\|_\mathcal{H}^2\bigr\}
     &\overset{\eqref{CorrectionS}}{=} &  \frac{1}{2}\big\{ \|{\xi}^k -\tilde{\xi}^k\|_\mathcal{H}^2  -   \|({\xi}^{k} -\tilde{\xi}^k) -\alpha\mathcal{M}(\xi^k - \tilde{\xi}^k)  \|_\mathcal{H}^2 \bigr\} \nn \\
 & \overset{\eqref{Matrix-GQ}}{=}& \frac{1}{2}\big\{  2 \alpha({\xi}^{k}- \tilde{\xi}^k) ^T\mathcal{Q}(\xi^k - \tilde{\xi}^k)
       -\alpha^2\|\mathcal{M}(\xi^k - \tilde{\xi}^k)  \|_{\mathcal{H}}^2  \bigr\} \nn \\
       & \overset{\eqref{Matrix-GQ}}{=} & \frac{1}{2}(\xi^k - \tilde{\xi}^k)^T[\alpha(\mathcal{Q}^T+\mathcal{Q})-\alpha^2\mathcal{D}]  (\xi^k - \tilde{\xi}^k) \nn  \\
  &  \overset{\eqref{PQP}}{=} &  \frac{1}{2}\alpha(2-\alpha) \|{\xi}^k - \tilde{\xi}^k  \|_\mathcal{D}^2.
  \end{eqnarray}
Combining \eqref{LEM-M-b}, \eqref{LEM-M-c} and \eqref{LEM-M-d}, we obtain the equality \eqref{LEM-M-a}. Then, the first assertion \eqref{HauptA1} follows immediately by substituting \eqref{LEM-M-a} into \eqref{M-PRE-1}. Furthermore, setting  $w$ in \eqref{HauptA1} as any fixed $w^*\in \Omega^*$,  we get
\begin{equation}\label{HauptA2}
\|\xi^{k}-\xi^*\|_{\mathcal{H}}^2 -  \|\xi^{k+1}-\xi^*\|_{\mathcal{H}}^2 \ge 2\alpha\{\theta(\tilde{x}^k) -\theta(x^*)  +  (\tilde{w}^k-w^*)^T F(w^\ast)\} +\alpha(2-\alpha) \|{\xi}^k - \tilde{\xi}^k  \|_{\mathcal{D}}^2.
\end{equation}
Also, it follows from $w^*\in \Omega^*$ and \eqref{OVI} that
$$
\theta(\tilde{x}^k)-\theta(x^*) +   (\tilde{w}^k-w^*)^T F(w^*) \ge 0.
$$
This leads to the second assertion \eqref{XIcontr} immediately, and the proof is complete.
\end{proof}
\begin{remark}\label{ascent}
Setting $w$ as an arbitrary $w^\ast\in\Omega^\ast$ in \eqref{M-PRE-1}, we have
\begin{equation}\label{Ineq-A2}
(\tilde{\xi}^k -\xi^*)^T \mathcal{Q} (\xi^k-\tilde{\xi}^k) \geq \theta(\tilde{x}^k) -\theta(x^*)  + (\tilde{w}^k - w^*)^T F(w^*) \overset{\eqref{OVI}}{\ge}  0.
\end{equation}
Applying the equalities $\tilde{\xi}^k -\xi^*=\xi^k-\xi^\ast-(\xi^k-\tilde{\xi}^k)$ and $2\xi^T\mathcal{Q}\xi=\xi^T(\mathcal{Q}^T+\mathcal{Q})\xi$ to \eqref{Ineq-A2}, we get
\begin{equation}\label{Ineq-C}
  ({\xi}^k -\xi^*)^T \mathcal{Q} (\xi^k-\tilde{\xi}^k) \ge (\xi^k-\tilde{\xi}^k)^T\mathcal{Q}(\xi^k-\tilde{\xi}^k) = \frac{1}{2}(\xi^k-\tilde{\xi}^k)^T(\mathcal{Q}^T+\mathcal{Q})(\xi^k-\tilde{\xi}^k) \overset{\eqref{PQP}}{=}\|\xi^k -\tilde{\xi}^k\|_\mathcal{D}^2.
\end{equation}
Then, with the matrices $\mathcal{H}$ and $\mathcal{M}$ in \eqref{Matrix-M}, we obtain
\begin{eqnarray*}
 \lefteqn{\Big\langle\nabla\big(\frac{1}{2}\|\xi -\xi^*\|_\mathcal{H}^2\big)\big|_{\xi=\xi^k},-\mathcal{M} (\xi^k-\tilde{\xi}^k)\Big\rangle} \\
  &=& -({\xi}^k -\xi^*)^T \mathcal{H}\mathcal{M} (\xi^k-\tilde{\xi}^k) \overset{\eqref{Matrix-GQ}}{=}-({\xi}^k -\xi^*)^T \mathcal{Q} (\xi^k-\tilde{\xi}^k) \overset{\eqref{Ineq-C}}{\leq} -\|\xi^k -\tilde{\xi}^k\|_\mathcal{D}^2.
\end{eqnarray*}
This indicates that $d(\xi^k,\tilde{\xi}^k):=-\mathcal{M}(\xi^k - \tilde{\xi}^k)$ can be regarded as a descent direction along which the proximity to the solution set can be reduced.
\end{remark}

Now we are ready to show the global convergence of the rank-two relaxed parallel splitting version of the ALM  \eqref{Rank2M}.

\begin{theorem}\label{convergence}
The sequence $\{\xi^k\}$ generated by the proposed new algorithm \eqref{Rank2M} for \eqref{A-Problem-M} converges to some $\xi^\infty\in\Xi^\ast$,  where $\xi$ and $\Xi^\ast$ are defined in \eqref{Xi-notation} and \eqref{Xi-notation1}, respectively.
\end{theorem}
\begin{proof}
First of all, it follows from the inequality \eqref{XIcontr} that the sequence $\{\xi^k\}$ is bounded. Adding \eqref{XIcontr} over $k=0,\ldots,\infty$, we have
$$\sum_{k=0}^\infty\|\xi^k - \tilde{\xi}^k\|_\mathcal{D}^2\leq\frac{1}{\alpha(2-\alpha)}\| \xi^0 - \xi^*  \|_{\mathcal{H}}^2.$$
Considering the monotone convergence principle for the sequence $\{s_k:=\sum_{j=0}^k\|\xi^j - \tilde{\xi}^j\|_\mathcal{D}^2\}_{k\geq1}$, we obtain
\begin{equation}\label{limtbound}
  \lim_{k\rightarrow\infty}\|\xi^k - \tilde{\xi}^k\|_\mathcal{D}^2=0.
\end{equation}
The sequence $\{\tilde{\xi}^k\}$ is thus also bounded. Since $A_i$'s are assumed to be full column-rank in \eqref{A-Problem-M}, the sequence $\{\tilde{w}^k\}$ is bounded. Let $w^\infty$ be a cluster point of $\{\tilde{w}^k\}$ and $\{\tilde{w}^{k_j}\}$ be a subsequence converging to $w^\infty$. Recall that $\{\tilde{\xi}^k\}$ and $\{\tilde{\xi}^{k_j}\}$ are associated with $\{\tilde{w}^k\}$ and $\{\tilde{w}^{k_j}\}$, respectively. Then, it follows from \eqref{Prediction-w} that
$$\theta(x)-\theta(\tilde{x}^{k_j})+(w-\tilde{w}^{k_j})^TF(\tilde{w}^{k_j})\geq(\xi-\tilde{\xi}^{k_j})^T\mathcal{Q}(\xi^{k_j}-\tilde{\xi}^{k_j}),\quad \forall \; w\in\Omega.$$
According to \eqref{limtbound} and the continuity of $\theta$ and $F$, we have
$$w^\infty\in\Omega,\quad \theta(x) -\theta(x^\infty) +(w-w^{\infty})^T F(w^{\infty}) \ge 0, \quad \forall \;  w\in\Omega.$$
This means that $w^\infty$ is a solution point of the VI \eqref{VI1}, and hence $\xi^\infty:=Pw^\infty\in\Xi^\ast$. Furthermore, according to \eqref{XIcontr}, we have
\begin{equation}\label{III}
\begin{aligned}
  \|\xi^{k+1}-\xi^\infty\|_{\mathcal{H}}^2\leq\|\xi^k-\xi^\infty\|_{\mathcal{H}}^2.
  \end{aligned}
\end{equation}
Thus, the sequence $\{\|\xi^k - \xi^{\infty}\|_{\mathcal{H}}^2\}_{k \geq 0}$ is nonincreasing, and it is bounded away below from zero. Also, it follows from $\lim_{j\rightarrow\infty}\tilde{\xi}^{k_j}=\xi^\infty$ and \eqref{limtbound} that $\lim_{j\rightarrow\infty}\xi^{k_j}=\xi^\infty$. Therefore, we have $\lim_{k\rightarrow\infty}\xi^k=\xi^\infty\in\Xi^\ast$. The proof is complete.
\end{proof}

\begin{remark}
It follows from \eqref{Pall-XL} that $\sum_{i=1}^pA_ix_i^k-b=\frac{1}{\beta}(\lambda^k-\tilde{\lambda}^k)$. Hence, \eqref{limtbound} implies that the residual of the equality constraints in \eqref{A-Problem-M}, i.e., $\|\sum_{i=1}^pA_ix_i^k-b\|$, converges to 0 when $k\to \infty$.
\end{remark}

\section{Convergence rate}\label{Sec-rate}

\setcounter{equation}{0}
\setcounter{remark}{0}
\setcounter{proposition}{0}

In this section, we derive the worst-case $O(1/N)$ convergence rate in both the ergodic and point-wise  (a.k.a., nonergodic) senses for the rank-two relaxed parallel splitting version of the ALM  \eqref{Rank2M}, where $N$ denotes the iteration counter. We mainly follow the techniques in previous works \cite{Beck,he2017convergence,Heuniform2021,HY-SINUM,HY-NM} to derive the convergence rate.

\subsection{Ergodic convergence rate}

We first derive the worst-case $O(1/N)$ convergence rate in the ergodic sense for the new algorithm \eqref{Rank2M} in terms of the reduction of the objective function value and the residual of the equality constraints of the model \eqref{A-Problem-M}.

\begin{theorem}\label{erogic}
Let $\{\xi^k\}$ be the sequence generated by the proposed new algorithm \eqref{Rank2M} for \eqref{A-Problem-M} and $\tilde{w}^k$ be the output of the parallel splitting ALM step \eqref{Pall-XL},  and let
\begin{equation}\label{averagep}
{\bar{w}_{\!_N}} := \frac{1}{N+1} \sum_{k=0}^N \tilde{w}^k
\end{equation}
for any integer $N>0$. Suppose that $(x^\ast,\lambda^\ast)\in\Omega^\ast$ is a saddle point of the Lagrangian function \eqref{lagra}. Then, for any $\alpha\in(0,2)$, it holds that
\begin{equation}\label{ergodic-rate}
\theta(\bar{x}_{\!_N}) -\theta(x^*) \leq  O(1/N) \quad \hbox{and} \quad \|\mathcal{A}\bar{x}_{\!_N} -b\|  \leq  O(1/N).
\end{equation}
\end{theorem}
\begin{proof}
Recall that the matrix $\mathcal{D}$ defined in \eqref{Matrix-QS} is positive definite. It follows from \eqref{HauptA1} that
\begin{equation}\label{crate1}
\tilde{w}^k \in \Omega, \;\; \theta(\tilde{x}^k) -  \theta(x) + (\tilde{w}^k-w)^TF(w) \leq  \frac{1}{2\alpha}\big\{\|\xi-\xi^k \|_\mathcal{H}^2 - \|\xi-\xi^{k+1}\|_\mathcal{H}^2 \big\}, \quad  \forall \; w\in \Omega.
\end{equation}
Adding \eqref{crate1} over $k=0,1,\ldots,N$, we get
$$
  \sum_{k=0}^N\theta(\tilde{x}^k)  - (N+1)\theta(x) + \Big(\sum_{k=0}^N\tilde{w}^k- (N+1)w\Big)^TF(w) \leq  \frac{1}{2\alpha}\|\xi-\xi^0 \|_\mathcal{H}^2, \quad  \forall \; w\in \Omega.
$$
With the notation ${\bar{w}_{\!_N}}$ defined in \eqref{averagep}, the above inequality can be rewritten as
\begin{equation}\label{crate2}
  \frac{1}{N+1}\sum_{k=0}^N\theta(\tilde{x}^k)  - \theta(x) + ({\bar{w}_{\!_N}}- w)^TF(w) \leq  \frac{1}{2\alpha(N+1)}\|\xi-\xi^0 \|_\mathcal{H}^2, \quad  \forall \; w\in \Omega.
\end{equation}
Since $\Omega$ is a convex set and $\tilde{w}^k\in\Omega$ for all $k\geq0$, we have ${\bar{w}_{\!_N}}\in\Omega$.  On the other hand,  it follows from the convexity of $\theta$ that
\begin{equation}\label{crate3}
   \theta({\bar{x}_{\!_N}})=\theta\Bigl(\frac{1}{N+1} \sum_{k=0}^N \tilde{x}^k\Bigr)\leq\frac{1}{N+1}\sum_{k=0}^N\theta(\tilde{x}^k).
\end{equation}
Substituting \eqref{crate3} into \eqref{crate2}, we obtain
$$
{\bar{w}_{\!_N}}\in\Omega,\quad  \theta(\bar{x}_{\!_N}) -\theta(x)  + (\bar{w}_{\!_N}-w)^T F(w) \le \frac{1}{2\alpha(N+1)}\|\xi-\xi^0\|_\mathcal{H}^2, \quad \forall \;w\in\Omega.
$$
Because of the skew-symmetry of $F$ (see \eqref{Skew-S}), the above inequality is equivalent to
\begin{equation}\label{crate4}
{\bar{w}_{\!_N}}\in\Omega,\quad  \theta(\bar{x}_{\!_N}) -\theta(x)  + (\bar{w}_{\!_N}-w)^T F(\bar{w}_{\!_N}) \le \frac{1}{2\alpha(N+1)}\|\xi-\xi^0\|_\mathcal{H}^2, \quad \forall \;w\in\Omega.
\end{equation}
Plugging $x=x^\ast$ into \eqref{crate4} and using the notations defined in \eqref{Matrix-H-T2} and \eqref{averagep}, we obtain
\begin{eqnarray*}
&&\theta(\bar{x}_{\!_N}) -\theta(x^*)-(\bar{x}_{\!_N}-x^*)^T\mathcal{A}^T \bar{\lambda}_{\!_N}+(\bar{\lambda}_{\!_N}-\lambda)^T(\mathcal{A}\bar{x}_{\!_N} -b)  \\[0.1cm]
&&\quad \leq \frac{1}{2\alpha(N+1)}\Big\{\|{\hat{\xi}}^*-\hat{\xi}^0\|_{\hat{\mathcal{H}}}^2+\frac{1}{\beta}(1+p)\|\lambda-\lambda^0\|^2\Big\}, \quad \forall \; \lambda \in \Re^m.
\end{eqnarray*}
Moreover, it follows from $\mathcal{A}x^*-b=0$ that
$$  \theta(\bar{x}_{\!_N}) -\theta(x^*)-\lambda^T(\mathcal{A}\bar{x}_{\!_N} -b) \leq \frac{1}{2\alpha(N+1)}\Big\{\|\hat{\xi}^*-\hat{\xi}^0\|_{\hat{\mathcal{H}}}^2+\frac{1}{\beta}(1+p)\|\lambda-\lambda^0\|^2\Big\}, \quad \forall \; \lambda \in \Re^m.$$
Since it holds that
$$
\|\lambda-\lambda^0\|^2\leq(\|\lambda\|+\|\lambda^0\|)^2,\quad \forall \; \lambda \in \Re^m, \; \lambda_0 \in \Re^m,
$$
we further obtain
\begin{equation}\label{crate5}
 \theta(\bar{x}_{\!_N}) -\theta(x^*)-\lambda^T(\mathcal{A}\bar{x}_{\!_N} -b) \leq \frac{1}{2\alpha(N+1)}\Big\{\|\hat{\xi}^*-\hat{\xi}^0\|_{\hat{\mathcal{H}}}^2+\frac{1}{\beta}(1+p)(\|\lambda\|+\|\lambda^0\|)^2\Big\}, \quad \forall \; \lambda \in \Re^m.
\end{equation}
Without loss of generality, suppose $\mathcal{A}\bar{x}_{\!_N} -b\neq0$. Taking $\lambda=-(\mathcal{A}\bar{x}_{\!_N} -b)/\|\mathcal{A}\bar{x}_{\!_N} -b\|$ in \eqref{crate5}, we have
$$
\theta(\bar{x}_{\!_N}) -\theta(x^*)+\|\mathcal{A}\bar{x}_{\!_N} -b\| \leq \frac{1}{2\alpha(N+1)}\Big\{ \|\hat{\xi}^*-\hat{\xi}^0\|_{\hat{\mathcal{H}}}^2+\frac{1}{\beta}(1+p)(1+\|\lambda^0\|)^2\Big\},
$$
which implies the assertion (\ref{ergodic-rate}) immediately.
\end{proof}

\subsection{Point-wise convergence rate}

Now, we establish a worst-case $O(1/N)$ convergence rate for the rank-two relaxed parallel splitting version of the ALM \eqref{Rank2M} in the point-wise sense. We start from a theorem indicating certain monotonicity of the sequence $\{\|\xi^k-{\xi}^{k+1}\|_\mathcal{H}^2$\}.

\begin{theorem}
Let $\{\xi^k\}$ be the sequence generated by the proposed new algorithm \eqref{Rank2M} for \eqref{A-Problem-M}, and $\mathcal{H}$ be the matrix defined in \eqref{Matrix-M}. Then, for any integer $k\geq0$ and $\alpha\in(0,2)$, we have
\begin{equation}\label{point-2}
\|\xi^{k+1}-{\xi}^{k+2}\|_\mathcal{H}^2\leq \|\xi^k-{\xi}^{k+1}\|_\mathcal{H}^2-\alpha(2-\alpha) \|(\xi^k-\tilde{\xi}^k)-(\xi^{k+1}-\tilde{\xi}^{k+1})\|_\mathcal{D}^2.
\end{equation}
\end{theorem}
\begin{proof}
Utilizing the identity $\|a\|_\mathcal{H}^2-\|b\|_\mathcal{H}^2=2a^T\mathcal{H}(a-b)-\|a-b\|_\mathcal{H}^2$ with $a=(\xi^k-{\xi}^{k+1})$ and $b=(\xi^{k+1}-{\xi}^{k+2})$, we obtain
\begin{eqnarray}\label{Bequaliy}
\|\xi^k-{\xi}^{k+1}\|_\mathcal{H}^2  - \|\xi^{k+1}-{\xi}^{k+2}\|_\mathcal{H}^2
  &=& 2(\xi^k-{\xi}^{k+1})^T\mathcal{H}\{(\xi^k-{\xi}^{k+1})-(\xi^{k+1}-{\xi}^{k+2})\} \nn \\
   & &  -\|(\xi^k-{\xi}^{k+1})-(\xi^{k+1}-{\xi}^{k+2})\|_\mathcal{H}^2.
\end{eqnarray}
Let us first bound the first term in the right-hand side of \eqref{Bequaliy} by a quadratic term. To this end, setting $w=\tilde{w}^{k+1}$ in \eqref{Prediction-w}, we have
\begin{equation}\label{CR1}
  \theta(\tilde{x}^{k+1})-\theta(\tilde{x}^k)+(\tilde{w}^{k+1}-\tilde{w}^k)^TF(\tilde{w}^k)\geq(\tilde{\xi}^{k+1}-\tilde{\xi}^k)^T\mathcal{Q}(\xi^k-\tilde{\xi}^k).
\end{equation}
Also, rewriting the inequality \eqref{Prediction-w} for the $(k+1)$-th iteration leads to
\begin{equation}\label{CR2}
  \theta(x)-\theta(\tilde{x}^{k+1})+(w-\tilde{w}^{k+1})^TF(\tilde{w}^{k+1})\geq(\xi-\tilde{\xi}^{k+1})^T\mathcal{Q}(\xi^{k+1}-\tilde{\xi}^{k+1}),\quad \forall \; w\in\Omega.
\end{equation}
Setting $w=\tilde{w}^k$ in \eqref{CR2}, we obtain
\begin{equation}\label{CR3}
  \theta(\tilde{x}^{k})-\theta(\tilde{x}^{k+1})+(\tilde{w}^{k}-\tilde{w}^{k+1})^TF(\tilde{w}^{k+1})\geq(\tilde{\xi}^{k}-\tilde{\xi}^{k+1})^T\mathcal{Q}(\xi^{k+1}-\tilde{\xi}^{k+1}).
\end{equation}
Adding \eqref{CR1} and \eqref{CR3},  and combining with the monotonicity of $F$  (see \eqref{Skew-S}), we have
\begin{equation}\label{noneg1}
(\tilde{\xi}^{k}-\tilde{\xi}^{k+1})^T\mathcal{Q}\{(\xi^k-\tilde{\xi}^k)-(\xi^{k+1}-\tilde{\xi}^{k+1})\}\geq(\tilde{w}^{k}-\tilde{w}^{k+1})^T(F(\tilde{w}^k)-F(\tilde{w}^{k+1}))\overset{\eqref{Skew-S}}{\equiv}0.
\end{equation}
Moreover, adding the term $\{(\xi^k-\tilde{\xi}^k)-(\xi^{k+1}-\tilde{\xi}^{k+1})\}^T\mathcal{Q}\{(\xi^k-\tilde{\xi}^k)-(\xi^{k+1}-\tilde{\xi}^{k+1})\}$ to both sides of \eqref{noneg1} and using $2\xi^T\mathcal{Q}\xi=\xi^T(\mathcal{Q}^T+\mathcal{Q})\xi\overset{\eqref{PQP}}{=}2\xi^T\mathcal{D}\xi$, we get
\begin{equation}\label{noneg2}
(\xi^k-\xi^{k+1})^T\mathcal{Q}\bigl\{(\xi^k-\tilde{\xi}^k)-(\xi^{k+1}-\tilde{\xi}^{k+1})\bigr\}\geq\|(\xi^k-\tilde{\xi}^k)-(\xi^{k+1}-\tilde{\xi}^{k+1})\|_{\mathcal{D}}^2.
\end{equation}
Meanwhile, note that the left-hand side of \eqref{noneg2} can be rewritten as
\begin{eqnarray*}
(\xi^k-\xi^{k+1})^T\mathcal{Q}\{(\xi^k-\tilde{\xi}^k)-(\xi^{k+1}-\tilde{\xi}^{k+1})\}&\overset{\eqref{Matrix-GQ}}{=}&(\xi^k-\xi^{k+1})^T\mathcal{H}\mathcal{M}\{(\xi^k-\tilde{\xi}^k)-(\xi^{k+1}-\tilde{\xi}^{k+1})\} \\
   &\overset{\eqref{CorrectionS}}{=}& \frac{1}{\alpha}(\xi^k-\xi^{k+1})^T\mathcal{H}\{(\xi^k-{\xi}^{k+1})-(\xi^{k+1}-{\xi}^{k+2})\}.
\end{eqnarray*}
We thus obtain
\begin{equation}\label{point-1}
  2(\xi^k-\xi^{k+1})^T\mathcal{H}\{(\xi^k-{\xi}^{k+1})-(\xi^{k+1}-{\xi}^{k+2})\}\geq 2\alpha\|(\xi^k-\tilde{\xi}^k)-(\xi^{k+1}-\tilde{\xi}^{k+1})\|_{\mathcal{D}}^2.
\end{equation}
Furthermore, substituting \eqref{point-1} into \eqref{Bequaliy}, we have
\begin{eqnarray*}
  \lefteqn{ \|\xi^k-{\xi}^{k+1}\|_\mathcal{H}^2  - \|\xi^{k+1}-{\xi}^{k+2}\|_\mathcal{H}^2} \\
   &=&  2(\xi^k-{\xi}^{k+1})^T\mathcal{H}\bigl\{(\xi^k-{\xi}^{k+1})-(\xi^{k+1}-{\xi}^{k+2})\bigr\}  -\|(\xi^k-{\xi}^{k+1})-(\xi^{k+1}-{\xi}^{k+2})\|_\mathcal{H}^2 \\
   &\overset{\eqref{point-1}}{\geq}& 2\alpha\|(\xi^k-\tilde{\xi}^k)-(\xi^{k+1}-\tilde{\xi}^{k+1})\|_{\mathcal{D}}^2 -\|(\xi^k-{\xi}^{k+1})-(\xi^{k+1}-{\xi}^{k+2})\|_\mathcal{H}^2 \\
   &\overset{\eqref{CorrectionS}}{=}& 2\alpha\|(\xi^k-\tilde{\xi}^k)-(\xi^{k+1}-\tilde{\xi}^{k+1})\|_{\mathcal{D}}^2-\alpha^2\|\mathcal{M}\{(\xi^k-\tilde{\xi}^k)-(\xi^{k+1}-\tilde{\xi}^{k+1})\}\|_\mathcal{H}^2 \\
   &\overset{\eqref{Matrix-GQ}}{=}& \alpha(2-\alpha) \|(\xi^k-\tilde{\xi}^k)-(\xi^{k+1}-\tilde{\xi}^{k+1})\|_\mathcal{D}^2,
\end{eqnarray*}
and the proof is complete.
\end{proof}

Then, a worst-case $O(1/N)$ convergence rate for the proposed new algorithm \eqref{Rank2M} in the point-wise sense can be proved.
\begin{theorem}\label{last-theorem}
Let $\xi$ and $\Xi^\ast$ be defined in \eqref{Xi-notation} and \eqref{Xi-notation1}, respectively, and let $\{\xi^k\}$ be the sequence generated by the proposed new algorithm \eqref{Rank2M} for \eqref{A-Problem-M}, and $\mathcal{H}$ be the matrix defined in \eqref{Matrix-M}. Then, for any integer $N>0$ and $\alpha\in(0,2)$,  we have
\begin{equation}\label{keyin1}
\|\xi^N- {\xi}^{N+1}\|_\mathcal{H}^2\leq\frac{\alpha}{(2-\alpha)(N+1)}\|\xi^0-\xi^\ast\|_\mathcal{H}^2, \quad  \forall  \; \xi^* \in \Xi^*.
\end{equation}
\end{theorem}
\begin{proof}
According to $\mathcal{D}=\mathcal{M}^T\mathcal{H}\mathcal{M}$ (see \eqref{Matrix-GQ}) and \eqref{CorrectionS}, we can rewrite \eqref{XIcontr} as
\begin{eqnarray}\label{THM-H-C-01}
 \|\xi^{k+1} - \xi^*\|_{\mathcal{H}}^2 &\le&   \| \xi^k - \xi^*  \|_{\mathcal{H}}^2   - \alpha(2-\alpha)  \|\xi^k - \tilde{\xi}^k\|_\mathcal{D}^2  \nonumber\\
 &\overset{\eqref{Matrix-GQ}}{=}& \|\xi^k  - \xi^*\|_\mathcal{H}^2 - \alpha(2-\alpha)\|\mathcal{M}(\xi^k -\tilde{\xi}^k)\|_\mathcal{H}^2 \nonumber \\
 &\overset{\eqref{CorrectionS}}{=}& \|\xi^k  - \xi^*\|_\mathcal{H}^2 - \frac{1}{\alpha}(2-\alpha)\|\xi^k -\xi^{k+1}\|_\mathcal{H}^2, \quad  \forall  \; \xi^* \in \Xi^*.
\end{eqnarray}
Summing \eqref{THM-H-C-01} over $k=0,\ldots,N$, we have
$$\sum_{k=0}^{N}\frac{1}{\alpha}(2-\alpha)\|\xi^k-{\xi}^{k+1}\|_\mathcal{H}^2\leq\|\xi^0-\xi^\ast\|_\mathcal{H}^2, \quad  \forall  \; \xi^* \in \Xi^*.$$
Moreover, it follows from \eqref{point-2} that the sequence $\{\|\xi^k-{\xi}^{k+1}\|_\mathcal{H}^2\}$ is monotonically non-increasing. It thus holds that
$$\frac{1}{\alpha}(2-\alpha)(N+1)\|\xi^N-{\xi}^{N+1}\|_\mathcal{H}^2\leq\sum_{k=0}^{N}\frac{1}{\alpha}(2-\alpha)\|\xi^k-{\xi}^{k+1}\|_\mathcal{H}^2\leq\|\xi^0-\xi^\ast\|_\mathcal{H}^2, \quad  \forall  \; \xi^* \in \Xi^*,$$
which results in the assertion of this theorem immediately.
\end{proof}

\smallskip
Let $d:=\inf\big\{\|\xi^0-\xi^\ast\|_\mathcal{H}^2 \mid  \xi^* \in \Xi^*\big\}$.  Then, according to Theorem \ref{last-theorem}, we have
$$\|\xi^N-{\xi}^{N+1}\|_\mathcal{H}^2\leq\frac{\alpha d}{(2-\alpha)(N+1)}=O(1/N).$$
Recall the inequality \eqref{Prediction-w} and the fact $\mathcal{Q}=\mathcal{H}\mathcal{M}$ (see \eqref{Matrix-GQ}). Then, $\tilde{w}^k$ is a solution point of the VI \eqref{VI1} if and only if $\|\xi^k-{\xi}^{k+1}\|_\mathcal{H}^2=0$. Hence, the assertion (\ref{keyin1}) indicates a worst-case $O(1/N)$ convergence rate in the point-wise sense for the proposed new algorithm \eqref{Rank2M}.

\section{Numerical experiments}\label{sec5}
\setcounter{equation}{0}

In this section, we apply the proposed rank-two relaxed parallel splitting version of the ALM \eqref{Rank2M} to some application problems, and validate its efficiency by some numerical results. We particularly compare the new algorithm  \eqref{Rank2M} with the JSALM \eqref{PSALM-P}-\eqref{PSALM-C} and the PJALM \eqref{PJALMP}. Our codes were written in Python 3.9 and were  executed  in a Lenovo laptop with  2.20 GHz Intel Core i7-8750H CPU and 16 GB memory.

\subsection{Latent variable Gaussian graphical model selection}\label{sec5.1}

\subsubsection{Model}\label{sec5.1.1}
We first consider the latent variable Gaussian graphical model selection problem which was proposed in \cite{chandrasekaran2012latent}. Its model is
\begin{equation}\label{LVGGMS}
  \begin{aligned}
  \min &\;\; \Phi(X,Y,Z):=\langle X,C\rangle-\log\det(X)+\nu\|Y\|_1+\mu\,\hbox{tr}(Z)\\[0.2cm]
      \hbox{s.t.} & \;\;\;X-Y+Z=0,\;\;Z\succeq0,
  \end{aligned}
\end{equation}
where $C\in\Re^{n\times n}$ is the covariance matrix obtained from the observation,  $\nu>0$ and $\mu>0$ are given positive weight parameters,  $\|\cdot\|_1$ is the entry-wise $\ell_1$ norm, and $\hbox{tr}(\cdot)$ is the trace of a matrix. Clearly, the model \eqref{LVGGMS} is a $3$-block separable convex programming problem with matrix variables, but it can be also regarded as a special case of \eqref{A-Problem-M} with $p=3$ if the variables in (\ref{A-Problem-M}) are extended to matrices.

\subsubsection{Subproblems}\label{sec5.1.2}

When the proposed rank-two relaxed parallel splitting version of the ALM \eqref{Rank2M} is applied to \eqref{LVGGMS}, the $x_i$-subproblems in the parallel splitting ALM step \eqref{Pall-XL} can be specified as
\begin{subequations}\label{LV-P}
  \begin{numcases}{}
  \label{LV-P1} \tilde{X}^{k} = \arg\min\Big\{\langle X,C\rangle-\log\det(X)-\langle\Lambda^k,X\rangle+\frac{\beta}{2}\|X-X^k\|_F^2\;\big|\;X\in\Re^{n\times n}\Big\},\\
  \label{LV-P2} \tilde{Y}^{k} = \arg\min\Big\{\nu\|Y\|_1-\langle\Lambda^k,-Y\rangle+\frac{\beta}{2}\|Y-Y^k\|_F^2\;\big|\;Y \in\Re^{n\times n}\Big\},\\
  \label{LV-P3} \tilde{Z}^{k} = \arg\min\Big\{\mu\hbox{tr}(Z) -\langle\Lambda^k,Z\rangle+\frac{\beta}{2}\|Z-Z^k\|_2^2\;\big|\;Z\succeq0,\,Z\in\Re^{n\times n} \Big\}.
  \end{numcases}
  \end{subequations}

For the $X$-subproblem \eqref{LV-P1}, according to the first-order optimality condition, it suffices to solve the nonlinear equation system:
\begin{equation}\label{num}
  C-X^{-1}-\Lambda^k+\beta(X-X^k)=0.
\end{equation}
Multiplying $X$ to both sides of \eqref{num}, we have
\begin{equation}\label{num1}
  \beta X^2+(C-\beta X^k-\Lambda^k)X-I=0.
\end{equation}
Let $UDU^T=C-\beta X^k-\Lambda^k$ be the eigenvalue decomposition. Substituting it into \eqref{num1} and setting $P=U^TXU$, we have
$$\beta PP+D P-I=0,  \;\;  \hbox{and thus} \;\;  P_{ii}=\frac{1}{2\beta}\Big(-D_{ii}+\sqrt{D_{ii}^2+4\beta}\;\Big).$$
Hence, we obtain that $\tilde{X}^k=U\,\hbox{diag}(P)\, U^T$ is a solution of \eqref{num}.
For the $Y$-subproblem \eqref{LV-P2},  since
$$\tilde{Y}^k=\arg\min_Y
\Big\{\nu\|Y\|_1-\langle\Lambda^k,-Y\rangle+\frac{\beta}{2}\|Y-Y^k\|_F^2\Big\}
=\arg\min_Y\Big\{\|Y\|_1+\frac{\beta}{2\nu}\big\|Y-(Y^k-\frac{1}{\beta}\Lambda^k)\big\|_F^2\Big\},$$
its solution can be expressed exactly by the soft shrinkage operator defined in, e.g., \cite{Chen,tao2011recovering}. For the $Z$-subproblem \eqref{LV-P3}, note that $$\tilde{Z}^{k}=\arg\min_{Z\succeq0}\Big\{\mu\hbox{tr}(Z)-\langle\Lambda^k,Z\rangle+\frac{\beta}{2}\|Z-Z^k\|_2^2\Big\} =\arg\min_{Z\succeq0}\Big\{\big\|Z-(Z^k+\frac{1}{\beta}(\Lambda^k-\mu I))\big\|_2^2\Big\},$$
and let $\textstyle VD_1V^T=Z^k+\frac{1}{\beta}(\Lambda^k-\mu I)$ be an eigenvalue decomposition. Then, it is trivial to verify  that $\tilde{Z}^k=V\max\{D_1,0\}\, V^T$ is a solution of \eqref{LV-P3}, where $\max\{D_1,0\}$ is taken component-wisely.

\subsubsection{Settings}\label{sec6.1.3}

To simulate, we follow some standard ways (e.g., as elucidated on  \url{http://web.stanford.edu/~boyd/papers/admm/covsel/covsel_example.html}) to generate the covariance matrix $C$. More concretely, we first randomly generate a sparse matrix $U\in\Re^{n\times n}$ with sparsity parameter $s=1$\textperthousand, whose nonzero entries are set to $1$ with a uniform distribution; and set $A=U+U^T$ if $U+U^T\succ0$ or $A=U+U^T+1.1|\min(\hbox{eig}(U+U^T))|I_n$ otherwise. We then set $S=A^{-1}$ as the true covariance matrix, and compute $D$ via a multivariate normal distribution whose mean is $e_n$, covariance matrix is $S$ and total sample number is $10n$. Finally, we generate $C$ by calculating the covariance of $D$. In addition, we take $\nu=0.005$ and $\mu=0.05$ in the model \eqref{LVGGMS}. The stopping criterion is
\begin{equation}\label{error-LVGGMS}
\max\Big\{\|X^{k}-X^{k-1}\|_F,\|Y^{k}-Y^{k-1}\|_F,\|Z^{k}-Z^{k-1}\|_F, \|X^k-Y^k+Z^k\|_F \Big\}<10^{-10}.
\end{equation}

For the common parameter $\beta$, we fix it as $\beta=0.2$ for all these three algorithms. For other parameters, each of them is well tuned for different algorithms individually. Recall that convergence of the JSALM \eqref{PSALM-P}-\eqref{PSALM-C} is theoretically guaranteed for any $\alpha\in(0,2(1-\sqrt{p/(p+1)}))$ and that of the PJALM \eqref{PJALMP} is guaranteed for any $\tau>p-1$.  Since $p=3$, we choose the asymptotically largest values $\alpha=2(1-\sqrt{3/(3+1)})$ in the JSALM \eqref{PSALM-P}-\eqref{PSALM-C} and $\tau=2$ in the PJALM \eqref{PJALMP}, respectively. For the new algorithm \eqref{Rank2M}, we choose $\alpha=1.5$, which is an empirically probed value with satisfactory numerical performance. Moreover, we set $(X^0,Y^0,Z^0,\Lambda^0)=(I_{n},2I_{n},I_{n},\mathbf{0}_{n\times n})$ for all the algorithms under comparison.

\subsubsection{Numerical results}\label{sec6.1.4}

In Table \ref{Ta1}, iteration numbers (``Iter"), computing {time} in seconds (``CPU(s)"), and values of objective function at the last iteration (``$\Phi(k)$") are reported for various sizes of $n$. Recall that all the algorithms under comparison allow to solve their $x_i$-subproblems in parallel. We count the sum of all its $x_i$-subproblems for each algorithm. We also test each scenario 5 times, and report the average of computing time to try to avoid the effect of natural oscillations of computing environment. According to Table \ref{Ta1}, the new algorithm \eqref{Rank2M} performs much more efficiently than the JSALM \eqref{PSALM-P}-\eqref{PSALM-C} and the PJALM \eqref{PJALMP}.

As mentioned, we choose $\alpha=1.5$ for the new algorithm (\ref{Rank2M}). In Figure \ref{fig1}, we demonstrate the numerical performance of other values of $\alpha$ for \eqref{LVGGMS} with $n=100$ and $n=200$, respectively. It is seen from these results that $\alpha=1.5$ is a good choice of \eqref{Rank2M} for solving \eqref{LVGGMS}. Moreover, since all algorithms originated from the ALM, including the three ones under comparison, have the same penalty parameter $\beta$ which may affect numerical performances, we fix $n=100$ and test these algorithms with different values of $\beta$. In Figure \ref{fig2}, we plot the iteration numbers and computing time in seconds for 10 different values of $\beta$ equally distanced in $[0.1,1.0]$, from which efficiency of the new algorithm \eqref{Rank2M} is further shown for different values of $\beta$.

\begin{table}[H]
\caption{Numerical results of \eqref{LVGGMS} with different $n$.}
\vspace{0.15cm}
\centering
\begin{tabular}{l ccccccc cccccccccc}
\toprule
 \multirow{2}{*}{$n$} &      \multicolumn{3}{c}{JSALM}   &   \multicolumn{3}{c}{PJALM}  &  \multicolumn{3}{c}{New algorithm \eqref{Rank2M}}    \cr \cmidrule(lr){2-4} \cmidrule(lr){5-7}\cmidrule(lr){8-10}
                                           &   Iter      &   CPU(s)  &   $\Phi(k)$ &   Iter        &   CPU(s)  &   $\Phi(k)$ &  Iter    &   CPU(s)  &    $\Phi(k)$     \cr
\midrule
    $50$        &   875     &    1.39        &   13.48       &   810        &   1.26       &   13.48       &  133    &   0.20       &     13.48      \\
    $100$      &   806     &    4.36        &   28.85       &   731        &   4.05       &   28.85       &  144    &   0.81       &     28.85        \\
    $200$      &   746     &    22.95      &   66.32       &   697        &   20.21     &   66.32       &  159    &   4.98       &     66.32      \\
    $300$      &   584     &    39.32      &   115.40     &   583        &   37.17     &   115.40     &  245    &   15.87     &     115.40      \\
    $400$      &   707     &   102.95     &   100.36     &   586        &   81.99     &   100.36     &  259    &   36.45     &     100.36      \\
    $500$      &   803     &   192.95     &   36.39       &   665        &   152.12   &   36.39       &  164    &   37.55     &     36.39      \\
  \bottomrule
 \end{tabular}
 \label{Ta1}
\end{table}

\begin{figure}[H]
\centering
\subfigure[$n=100$]{
\includegraphics[width=8.5cm]{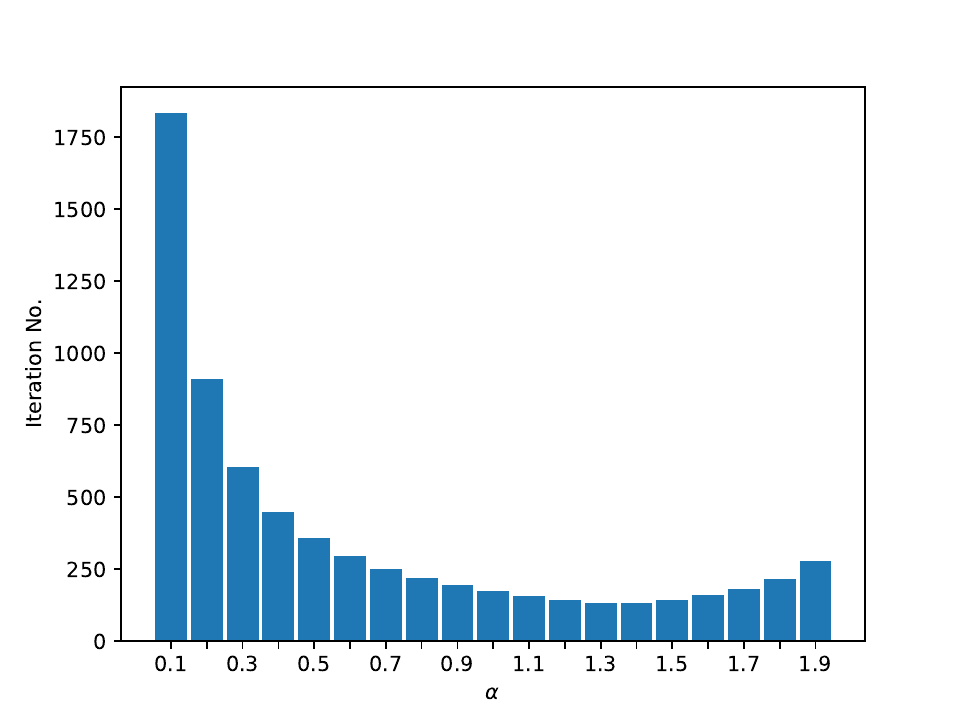}
}\hspace{-10mm}
\subfigure[$n=200$]{
\includegraphics[width=8.5cm]{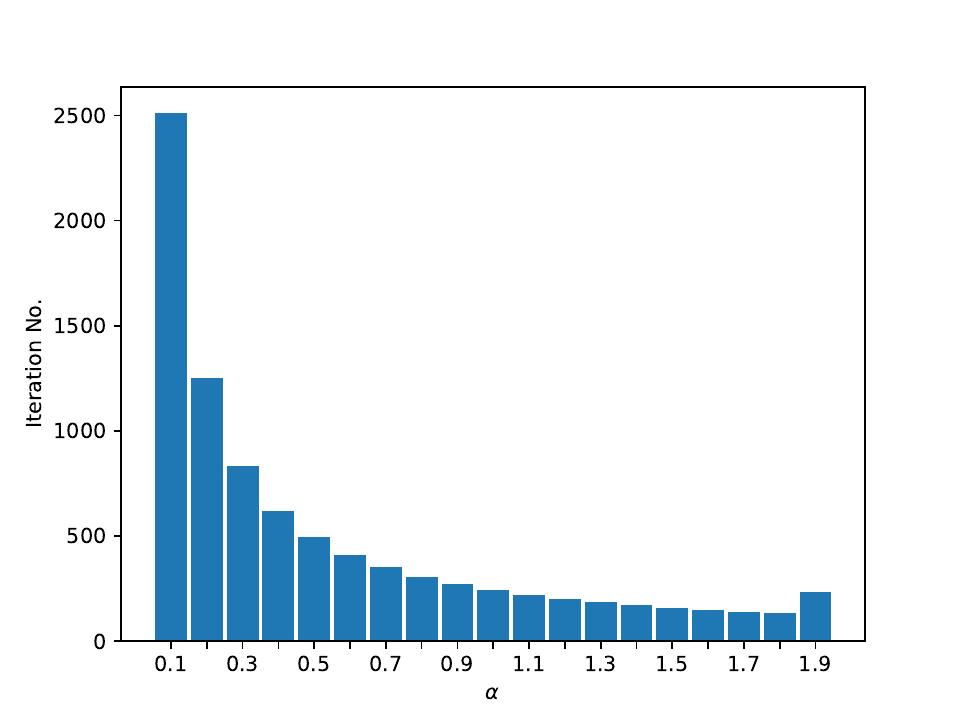}
}
\caption{Numerical results of \eqref{Rank2M} for \eqref{LVGGMS} with different $\alpha \in [0.1,1.9]$.}
\label{fig1}
\end{figure}

\begin{figure}[H]
\centering
\subfigure[$n=100$]{
\includegraphics[width=8.5cm]{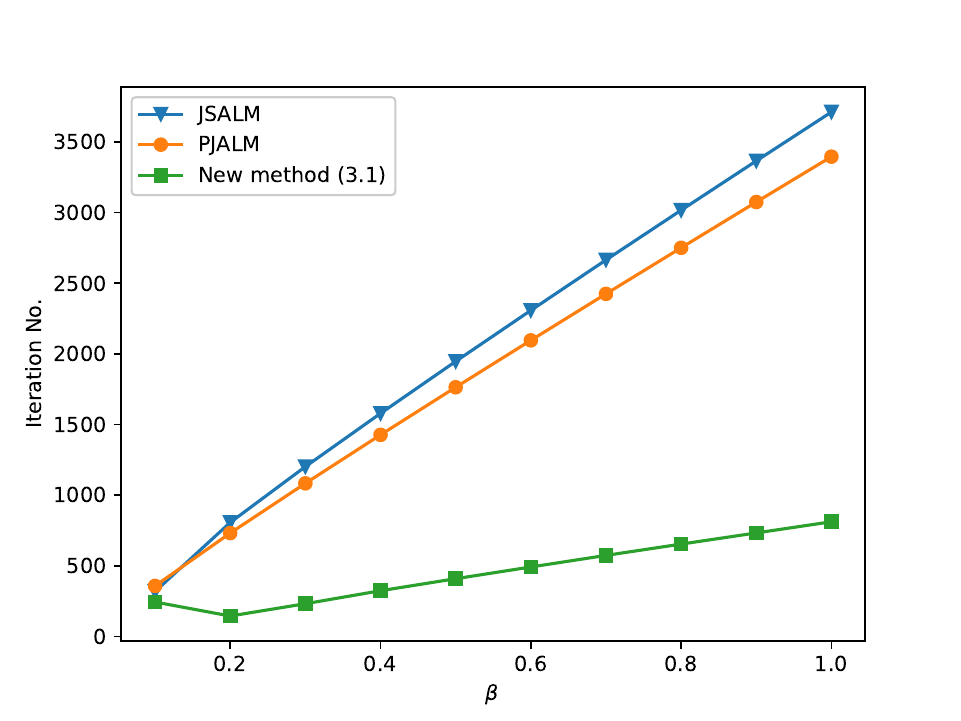}
}\hspace{-10mm}
\subfigure[$n=100$]{
\includegraphics[width=8.5cm]{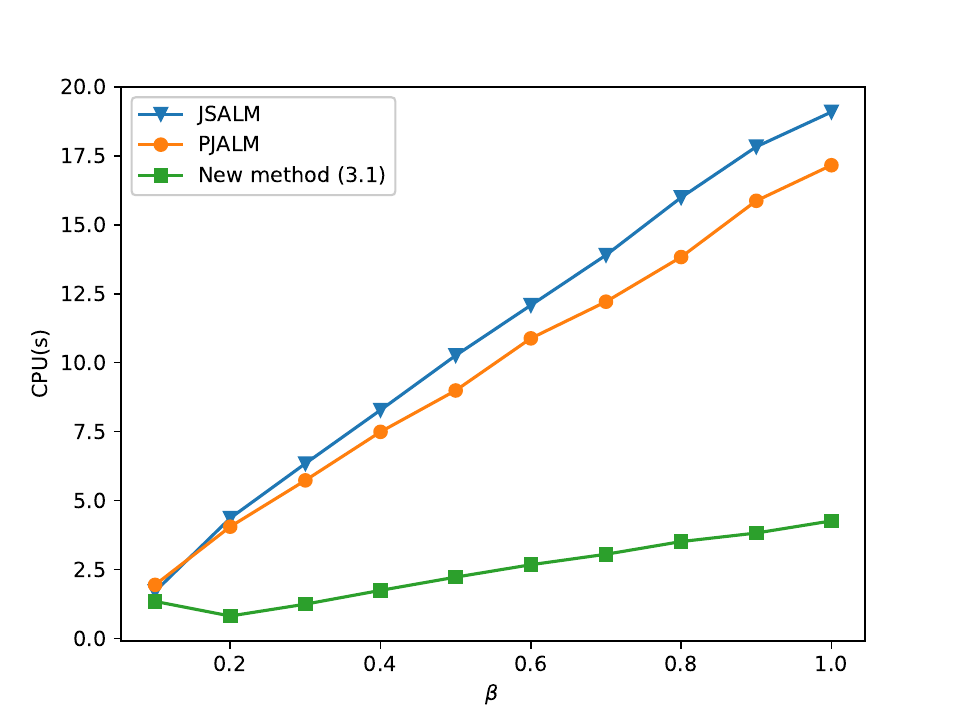}
}
\caption{Numerical results of \eqref{LVGGMS} with $n=100$ with different $\beta \in[0.1,1]$. }
\label{fig2}
\end{figure}

\subsection{Matrix decomposition problem}\label{sec5.2}

\subsubsection{Model}\label{sec5.2.1}

Then, we consider the matrix decomposition problem studied in \cite{Parikh2014}:
\begin{equation}\label{MDP}
  \begin{aligned}
  \min  &\;\; \Psi(X,Y,Z):=\|X\|_F^2+\mu\|Y\|_1+\nu\|Z\|_*  \\[0.2cm]
      \hbox{s.t.} & \;\;\;X+Y+Z=A,
  \end{aligned}
\end{equation}
where $A\in\Re^{m\times n}$ is a given data matrix, $\|\cdot\|_F$, $\|\cdot\|_1$ and $\|\cdot\|_*$ denote the $F$-norm, $l_1$-norm and nuclear norm, respectively, and $\mu>0$ and $\nu>0$ are trade-off parameters. The model \eqref{MDP} aims at decomposing $A$ into the sum of a matrix $X$ with small entries, a sparse matrix $Y$, and a low rank matrix $Z$. It is also closely related to the robust principal components analysis (RPCA) as studied in \cite{Can2011,tao2011recovering}. The model \eqref{MDP} is also a special case of \eqref{A-Problem-M} with $p=3$ and matrix variables.

\subsubsection{Subproblems}\label{sec5.2.2}

When the proposed rank-two relaxed parallel splitting version of the ALM \eqref{Rank2M} is applied to \eqref{MDP}, the $x_i$-subproblems in the parallel splitting ALM step \eqref{Pall-XL}  can be specified as
\begin{subequations}\label{M-Ps}
  \begin{numcases}{}
  \label{M-P1} \tilde{X}^{k} = \arg\min\Big\{\|X\|_F^2-\langle\Lambda^k,X\rangle+\frac{\beta}{2}\|X-X^k\|_F^2\;\big|\;X\in\Re^{m\times n}\Big\},\\
  \label{M-P2} \tilde{Y}^{k} = \arg\min\Big\{\mu\|Y\|_1-\langle\Lambda^k,Y\rangle+\frac{\beta}{2}\|Y-Y^k\|_F^2\;\big|\;Y\in\Re^{m\times n}\Big\},\\
  \label{M-P3} \tilde{Z}^{k} = \arg\min\Big\{\nu\|Z\|_\ast-\langle\Lambda^k,Z\rangle+\frac{\beta}{2}\|Z-Z^k\|_F^2\;\big|\;Z\in\Re^{m\times n}\Big\}.
  \end{numcases}
  \end{subequations}
For the $X$-subproblem \eqref{M-P1}, its solution point can be specified as
$$
\tilde{X}^{k} = (\Lambda^k + \beta X^k)/(2+\beta).
$$
For the  $Y$-subproblem \eqref{M-P2}, its solution can be represented exactly by the shrinkage operator defined in \cite{Chen}. For the $Z$-subproblem \eqref{M-P3}, it can be represented precisely by the proximal operator of the nuclear norm discussed in \cite{Can2013}.

\subsubsection{Settings}\label{sec6.2.3}

To simulate, we follow \cite{Parikh2014} to generate the data matrix $A=L+S+V$, where $L$, $S$ and $V$ are  rank-4 matrix, sparse matrix and noise matrix, respectively. More concretely, we set $L=L_1L_2$ with $L_1\in\Re^{m\times4}$ and $L_2\in\Re^{4\times n}$; entries of both $L_1$ and $L_2$ satisfy the normal distribution $\mathcal{N}(0,1)$; we generate the sparse matrix $S$  with density 0.05, with nonzero entries sampled uniformly from $\{-10, 10\}$; and entries of the noise matrix $V$ are sampled by the normal distribution $\mathcal{N}(0,10^{-3})$. In addition, we set $\mu=0.15\rho_1$ and $\nu=0.15\rho_2$, where $\rho_1$ and $\rho_2$ are the  entry-wise $\ell_\infty$ norm and the spectral norm of $A$, respectively. The stopping criterion is
\begin{equation}\label{error-matrix}
\max\Big\{\|X^{k}-X^{k-1}\|_F,\|Y^{k}-Y^{k-1}\|_F,\|Z^{k}-Z^{k-1}\|_F, \|X^k+Y^k+Z^k-A\|_F \Big\}<10^{-10}.
\end{equation}

For parameters, we follow the strategy mentioned in Section \ref{sec6.1.3}. That is, since $p=3$, we choose $\alpha=2(1-\sqrt{3/(3+1)})$ in the JSALM \eqref{PSALM-P}-\eqref{PSALM-C}, $\tau=2$ in the PJALM \eqref{PJALMP} and $\alpha=1.5$ in the new algorithm \eqref{Rank2M}. All these parameters are well tuned for each algorithm individually. For all the algorithms under comparison, $\beta=2.0$ and the initial point is $(X^0,Y^0,Z^0,\Lambda^0)=(\mathbf{0}_{m\times n},\mathbf{0}_{m\times n},\mathbf{0}_{m\times n},\mathbf{0}_{m\times n})$.

\subsubsection{Numerical results}\label{sec6.2.4}

In Table \ref{Ta2}, iteration numbers (``Iter"), computing time in seconds (``CPU(s)"), and values of the objective function at the last iteration (``$\Psi(k)$") are reported for various cases of $(m,n)$. Here, the computing time is counted by the same way as mentioned in Section \ref{sec6.1.4}. Table \ref{Ta2} further shows efficiency of the new algorithm \eqref{Rank2M}.  In Figure \ref{fig4}, we choose $(m,n)=(100,200)$ and $(m,n)=(200,500)$ for \eqref{MDP}, and plot the numerical result of  \eqref{Rank2M} with various values of $\alpha\in [0.1,1.9]$. It can be seen again that $\alpha=1.5$ is a satisfactory choice of \eqref{Rank2M} for solving \eqref{MDP}. Also, we fix $(m,n)=(200,500)$, and test the performance of these three algorithms with 10 different values of $\beta$ equally distributed in $[1,10]$ in Figure \ref{fig3}, from which efficiency of the new algorithm \eqref{Rank2M} is numerically verified again.

\begin{table}[H]
\caption{Numerical results of \eqref{MDP} with different $(m,n)$.}
\vspace{0.10cm}
\centering
\setlength{\tabcolsep}{1.6mm}{
\begin{tabular}{l ccccccc cccccccccc}
\toprule
\multirow{2}{*}{$(m,n)$} &      \multicolumn{3}{c}{JSALM}   &   \multicolumn{3}{c}{PJALM}  &  \multicolumn{3}{c}{New algorithm \eqref{Rank2M}}    \cr \cmidrule(lr){2-4} \cmidrule(lr){5-7}\cmidrule(lr){8-10}
                          &   Iter     &   CPU(s)  &   $\Psi(k)$        &   Iter        &   CPU(s)  &   $\Psi(k)$        &  Iter    &  CPU(s)   &    $\Psi(k)$     \cr
 \midrule
$(50,100)$        &   340     &    0.49    &   4281.15        &   245     &    0.35    &   4281.15        &   86       &    0.13    &   4281.15             \\
$(100,200)$      &   305     &    1.06    &   15523.07      &   212     &    0.70    &   15523.07      &   77       &    0.27    &   15523.07           \\
$(200,500)$      &   295     &    5.10    &   68194.13      &   200     &    3.13    &   68194.13      &   75       &    1.24    &   68194.13           \\
$(300,800)$      &   294     &    20.02  &   170884.53    &   197     &    12.70  &   170884.53    &   75       &    4.78    &   170884.53         \\
$(400,800)$      &   294     &    27.25  &   224700.41    &   198     &    17.36  &   224700.41    &   76       &    6.57    &   224700.41         \\
$(500,1000)$    &   295     &    45.04  &   356127.35    &   197     &    28.18  &   356127.35    &   77       &    11.03  &   356127.35         \\
  \bottomrule
 \end{tabular}}
 \label{Ta2}
\end{table}

\begin{figure}[H]
\centering
\subfigure[$(m,n)=(100,200)$]{
\includegraphics[width=8.5cm]{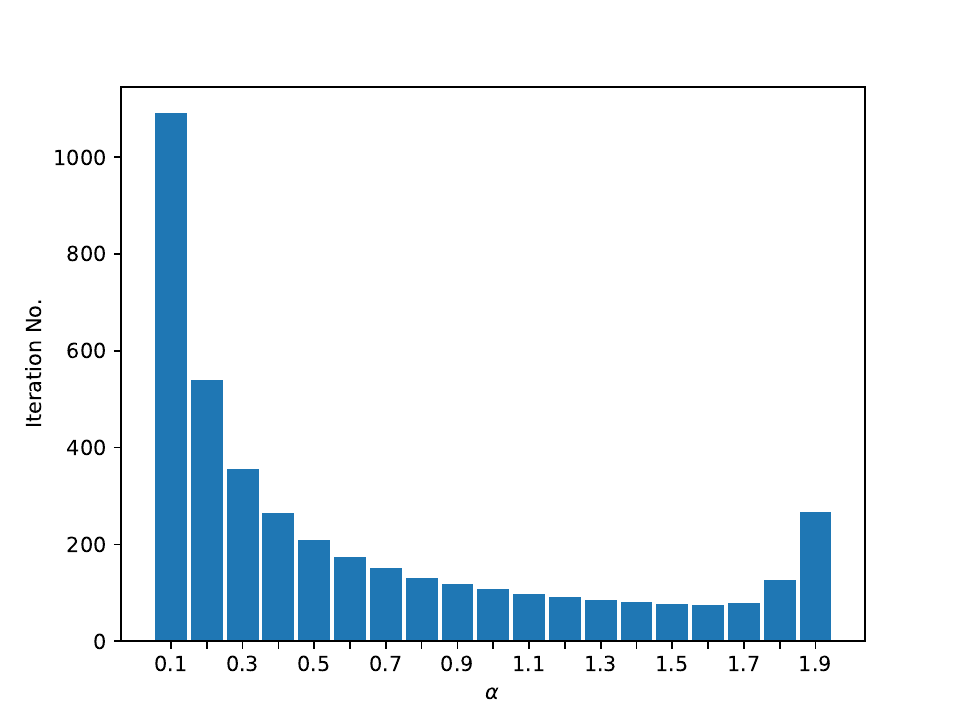}
}\hspace{-10mm}
\subfigure[$(m,n)=(200,500)$]{
\includegraphics[width=8.5cm]{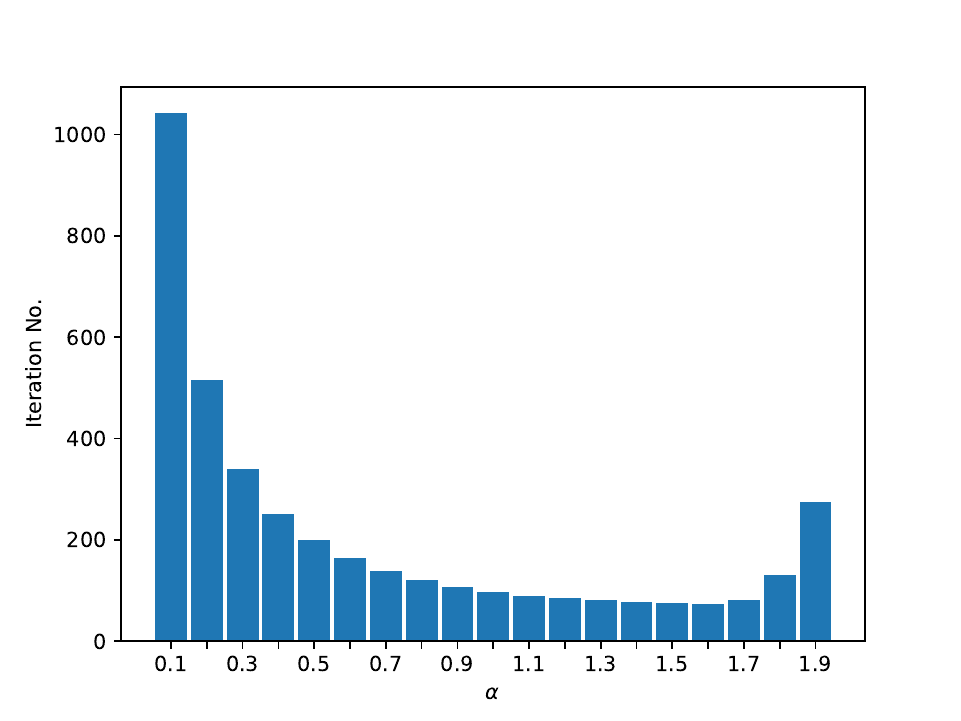}
}
\caption{Numerical results of \eqref{Rank2M} for \eqref{MDP} with different values of $\alpha \in [0.1,1.9]$.}
\label{fig4}
\end{figure}

\begin{figure}[H]
\centering
\subfigure[$(m,n)=(200,500)$]{
\includegraphics[width=8.5cm]{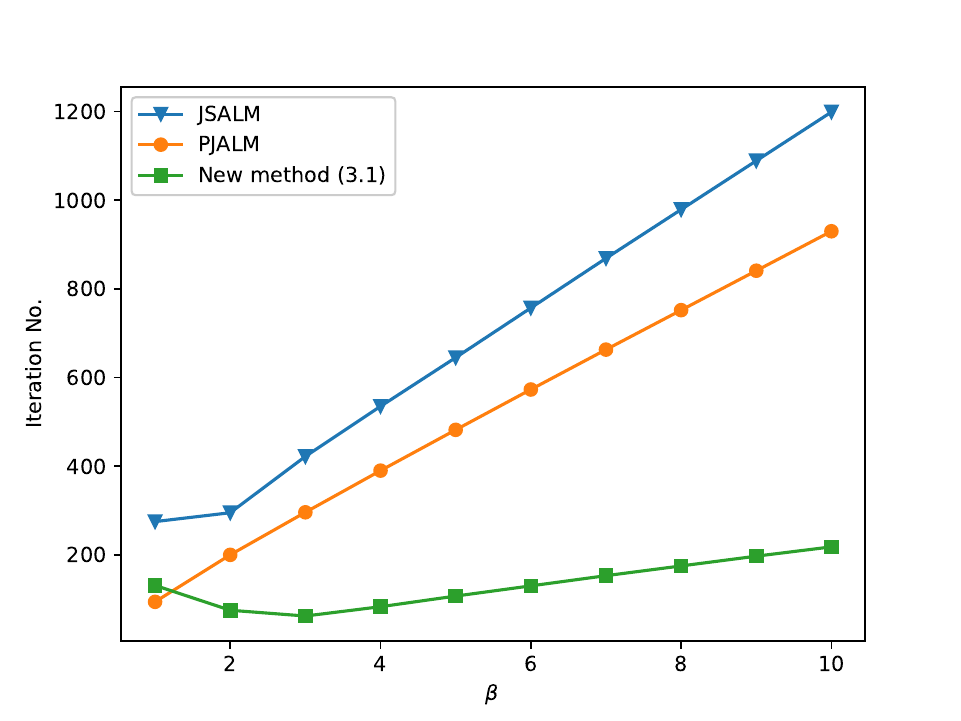}
}\hspace{-10mm}
\subfigure[$(m,n)=(200,500)$]{
\includegraphics[width=8.5cm]{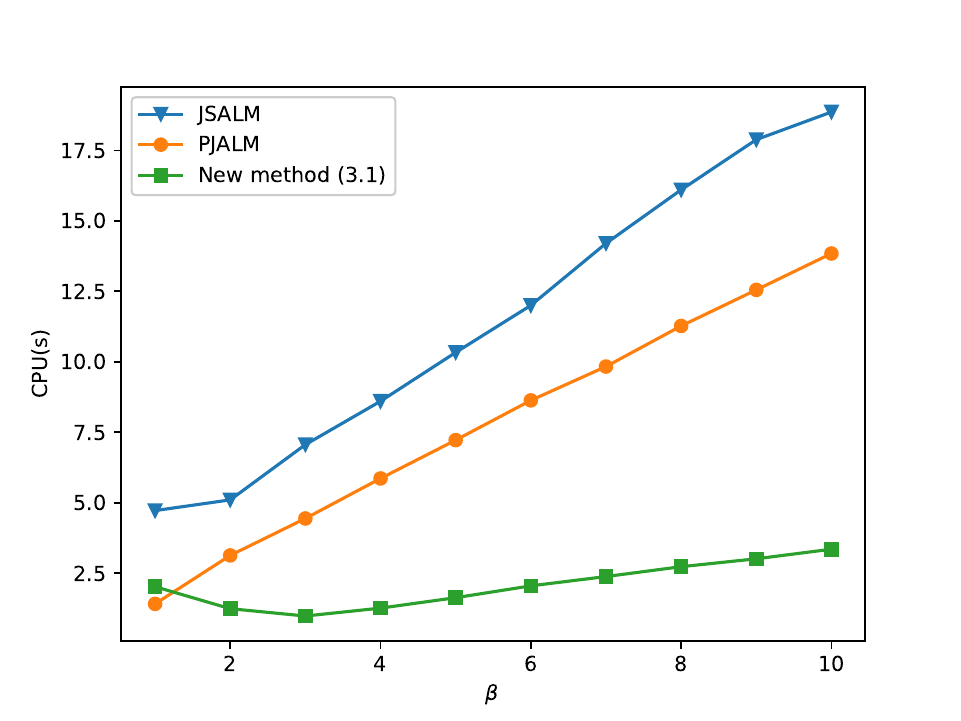}
}
\caption{Numerical results of \eqref{MDP} with $(m,n)=(200,500)$ with different $\beta \in[1,10]$.}
\label{fig3}
\end{figure}

\subsection{Exchange problem}\label{sec5.3}

\subsubsection{Model}\label{sec5.3.1}

Finally, we consider the exchange problem to minimize a function with a common objective among various agents. This problem arises in economics, and as discussed in \cite{boyd2010distributed,Deng2017}, its mathematical form is
\begin{equation}\label{Exchange}
  \min\bigg\{\sum_{i=1}^p\theta_i(x_i) \mid \sum_{i=1}^{p}x_i=\textbf{0}_n; \; x_i\in\Re^n, \, i=1,\ldots,p \bigg\},
\end{equation}
where $\theta_i:\Re^n\rightarrow \Re$ is a cost function corresponding to the agent $i$ for $i=1,\ldots,p$, and $p$ is the number of agents. We focus on a concrete example of \eqref{Exchange} as following:
\begin{equation}\label{Exchange1}
\min\bigg\{\frac{1}{2}\sum_{i=1}^p\|B_ix_i-c_i\|_2^2 \mid \sum_{i=1}^{p}x_i=\textbf{0}_n; \; x_i\in\Re^n, \, i=1,\ldots,p \bigg\},
\end{equation}
where $B_i\in\Re^{l\times n} \, (l<n)$ and $c_i\in\Re^l$ are given matrices and constant vectors for $i=1,\ldots,p$. The model (\ref{Exchange1}) is a special case of (\ref{A-Problem-M}) with $A_i=I_n$ for $i=1,\ldots,p$, and $b=\textbf{0}_n$. We also refer to \cite{Deng2017} for more discussions on the model \eqref{Exchange1}.

\subsubsection{Subproblems}\label{sec5.3.2}

When the new algorithm \eqref{Rank2M} is applied to \eqref{Exchange1}, it is easy to verify that the resulting iterative scheme is
$$\left\{\begin{array}{ccl}
      \tilde{x}_i^k & = & [\beta I_n +B_i^TB_i ]^{-1}(B_i^Tc_i + \lambda^k + \beta x_i^k), \; i=1,\ldots,p,  \\[0.2cm]
       x_i^{k+1}    & = & x_i^k - \alpha(x_i^k - \tilde{x}_i^k) - \frac{\alpha}{p+1}\sum_{j=1}^p\tilde{x}_j^k, \; i=1,\ldots,p,  \\[0.2cm]
      \lambda^{k+1} & = & \lambda^k - \frac{\alpha \beta}{p+1}\sum_{j=1}^p\tilde{x}_j^k.
 \end{array}\right.$$
All the subproblems can be solved easily with closed-form solutions.

\subsubsection{Settings}\label{sec5.3.3}

To simulate, we follow  \cite{Deng2017} and generate $x_i^\ast\in\Re^n\;(i=1,\ldots,p-1)$  randomly whose entries satisfy the normal distribution $\mathcal{N}(0,1)$, and set $x_p^\ast=-\sum_{i=1}^{p-1}x_i^\ast$, $B_i$ $(i=1,\ldots,p)$ are random Gaussian matrices, and each $c_i$ is computed by $c_i=B_ix_i^\ast$. It is obvious that the just-defined $x^\ast=(x_1^\ast,\ldots,x_p^\ast)$ is an optimal solution of \eqref{Exchange1} (may be not unique) and the minimum of the objective function of  \eqref{Exchange1} is 0. We consider the stopping criterion for \eqref{Exchange1} as
\begin{equation}\label{error-exchange}
\max\big\{\|x_i^k-x_i^{k-1}\|_2,\;i=1,\ldots,p,\;\; \|\textstyle\sum_{i=1}^px_i^k\|_2\big\}<10^{-5}.
\end{equation}
In addition, we set $n=50$ and $l=30$ in \eqref{Exchange1}, and choose $x_i^0=\textbf{0}_n\;(i=1,\ldots,p)$ and $\lambda^0=\textbf{0}_n$ as the initial iterate for all the algorithms under comparison.

The common parameter $\beta=1.0$ for all the three algorithms. For other individual parameters, we follow the same strategy as mentioned in Sections \ref{sec6.1.3} and \ref{sec6.2.3}, and set $\alpha=2(1-\sqrt{p/(p+1)})$ in the JSALM \eqref{PSALM-P}-\eqref{PSALM-C}, $\tau=p-1$ in the PJALM \eqref{PJALMP} for all $p$, and $\alpha=1.5$ in the new algorithm \eqref{Rank2M}.

\subsubsection{Numerical results}\label{sec5.3.4}

In Table \ref{Ta3}, iteration numbers (``Iter"), computing time in seconds (``CPU(s)"), and the errors (``Error") at the last iteration are reported for different values of $p$. Here, we define
\begin{equation}\label{model-error}
  \hbox{Error}(k):=\max\big\{\frac{1}{2}\textstyle\sum_{i=1}^p\|B_ix_i^k-c_i\|_2^2, \;\; \textstyle\|\sum_{i=1}^px_i^k\|_2\big\}.
\end{equation}
Efficiency of the new algorithm \eqref{Rank2M} is further demonstrated in Table \ref{Ta3} for \eqref{Exchange1}. Note that the computing time is counted by the same way as mentioned in Sections \ref{sec6.1.4} and \ref{sec6.2.4}. Compared with the previous examples \eqref{LVGGMS} and \eqref{MDP} which are both 3-block cases, the example \eqref{Exchange1} can have much larger number of blocks and the conditions \eqref{PSALM-step} and \eqref{PJALMC} become too restrictive (see, e.g.,  Table \ref{Ta4} for the accordingly computed step sizes with different $p$ for the JSALM \eqref{PSALM-P}-\eqref{PSALM-C}), and thus the convergence of the JSALM \eqref{PSALM-P}-\eqref{PSALM-C} and the PJALM \eqref{PJALMP} is substantially slowed down for large values of $p$. Since the new algorithm \eqref{Rank2M} does not have any additional condition depending on $p$, it performs very well even when $p$ is large, and for this case it can accelerate the JSALM \eqref{PSALM-P}-\eqref{PSALM-C} and the PJALM \eqref{PJALMP} significantly to a different scale. Moreover, its performance is very stable with respect to the value of $p$. These unique advantages make the new algorithm \eqref{Rank2M} very attractive to the case \eqref{A-Problem-M} with large $p$.

In Figure \ref{fig6}, we plot the numerical performance of \eqref{Rank2M} with various values of $\alpha \in [0.1,1.9]$ when $p=100$ and $p=200$ in \eqref{Exchange1}, respectively. This figure further justifies that $\alpha=1.5$ is a generally good choice of the new algorithm (\ref{Rank2M}) for various applications. In Figure \ref{fig5}, we report the performance of all the three algorithms for 10 different values of $\beta$ equally distributed in $[1,10]$ and further show efficiency of the new algorithm \eqref{Rank2M}.


\begin{table}[H]
\caption{Numerical results of \eqref{Exchange1} with different $p$.}
\vspace{0.15cm}
\centering
\begin{tabular}{l ccccccc cccccccccc}
\toprule
\multirow{2}{*}{$p$} &      \multicolumn{3}{c}{JSALM}   &   \multicolumn{3}{c}{PJALM}  &  \multicolumn{3}{c}{New algorithm \eqref{Rank2M}}    \cr \cmidrule(lr){2-4} \cmidrule(lr){5-7}\cmidrule(lr){8-10}
       &   Iter        &    CPU(s)   &  Error        &   Iter       &   CPU(s)  &   Error        &  Iter  & CPU(s) &      Error     \cr
  \midrule
$100$      &   3474    &    21.36       &   9.96e-6    &   476      &    2.79      &   8.88e-7    &  68    &   0.41    &     5.36e-6       \\
$200$      &   7227    &    87.25       &   9.99e-6    &   864      &    10.08    &   1.86e-6    &  63    &   0.74    &     6.20e-6       \\
$300$      &   11084  &    220.76     &   9.99e-6    &   1193    &    22.72    &   6.21e-6    &  62    &   1.21    &     7.60e-6       \\
$400$      &   15011  &    407.90     &   9.99e-6    &   1676    &    44.76    &   1.07e-5    &  62    &   1.67    &     5.23e-6       \\
$500$      &   18988  &    745.48     &   9.99e-6    &   2251    &    77.19    &   1.51e-5    &  62    &   2.17    &     5.09e-6       \\
$600$      &   23004  &    969.08     &   9.99e-6    &   2384    &   100.17   &   2.64e-5    &  62    &   2.57    &     5.50e-6       \\
$700$      &   27055  &    1330.69   &   9.99e-6    &   3437    &   171.08   &   2.46e-5    &  60    &   2.95    &     5.15e-6       \\
$800$      &   31133  &    1725.53   &   9.99e-6    &   2722    &   153.35   &   9.76e-5    &  61    &   3.39    &     5.71e-6       \\
$900$      &   35238  &    2208.85   &   9.99e-6    &   4175    &   266.19   &   2.73e-5    &  60    &   3.77    &     6.31e-6       \\
$1000$    &   39364  &    2786.61   &   9.99e-6    &   4307    &   306.53   &   7.24e-5    &  60    &   4.17    &     6.06e-6       \\
  \bottomrule
 \end{tabular}
 \label{Ta3}
\end{table}

\begin{table}[H]
\caption{Step size of the JSALM \eqref{PSALM-P}-\eqref{PSALM-C} for various values of $p$.}
\vspace{0.35cm}
 \centering
 \setlength{\tabcolsep}{2.0mm}{
 \begin{tabular}{l ccccccc cccccccccc}
  \toprule
      $p$    &   $5$  &  $10$   &  $100$   &   $500$  &   $1000$   &   $2000$   &   $5000$    \\
  \midrule
  $2\big(1- \sqrt{p/(p+1)}\big)$   &  0.1743   &     0.0931      &   0.0099      &    0.0020    &    9.99e-4    &   5.00e-4   &    2.00e-4        \\
  \bottomrule
 \end{tabular}}
 \label{Ta4}
\end{table}

\begin{figure}[H]
\centering
\subfigure[$p=100$]{
\includegraphics[width=8.6cm]{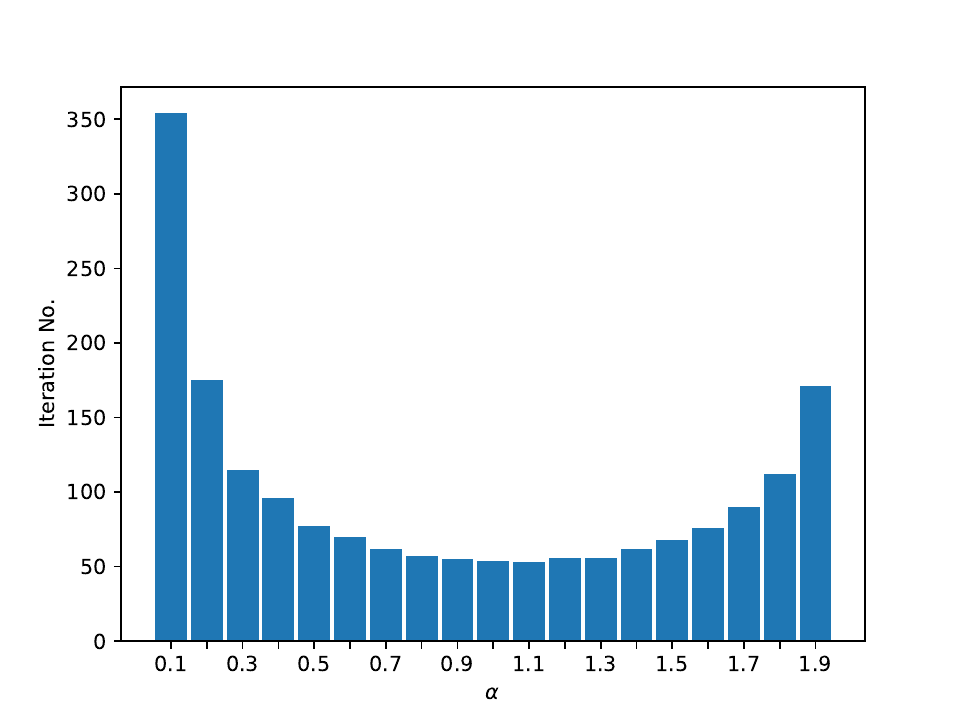}
}\hspace{-10mm}
\subfigure[$p=200$]{
\includegraphics[width=8.6cm]{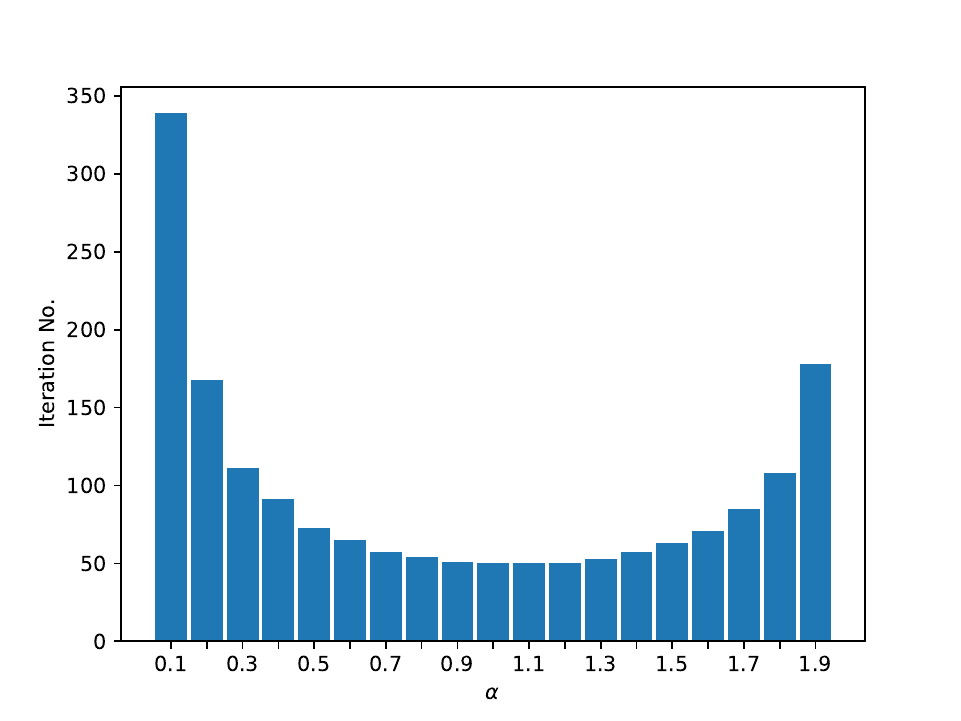}
}
\caption{Numerical results of \eqref{Rank2M} for \eqref{Exchange1} with different $\alpha \in [0.1,1.9]$.}
\label{fig6}
\end{figure}

\begin{figure}[H]
\centering
\subfigure[$p=100$]{
\includegraphics[width=8.6cm]{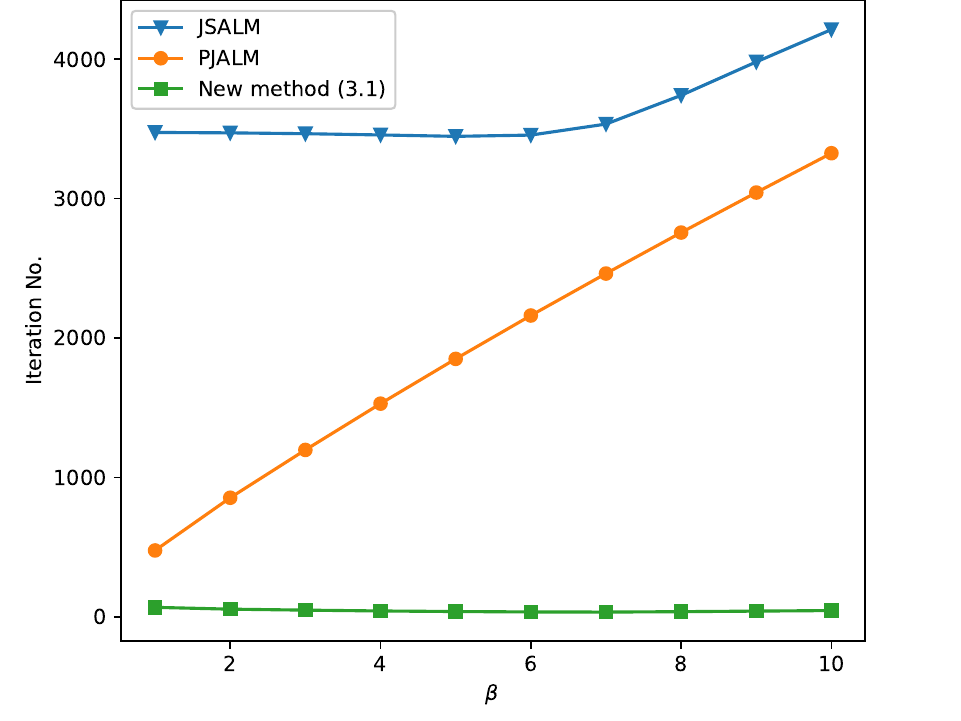}
}\hspace{-10mm}
\subfigure[$p=100$]{
\includegraphics[width=8.6cm]{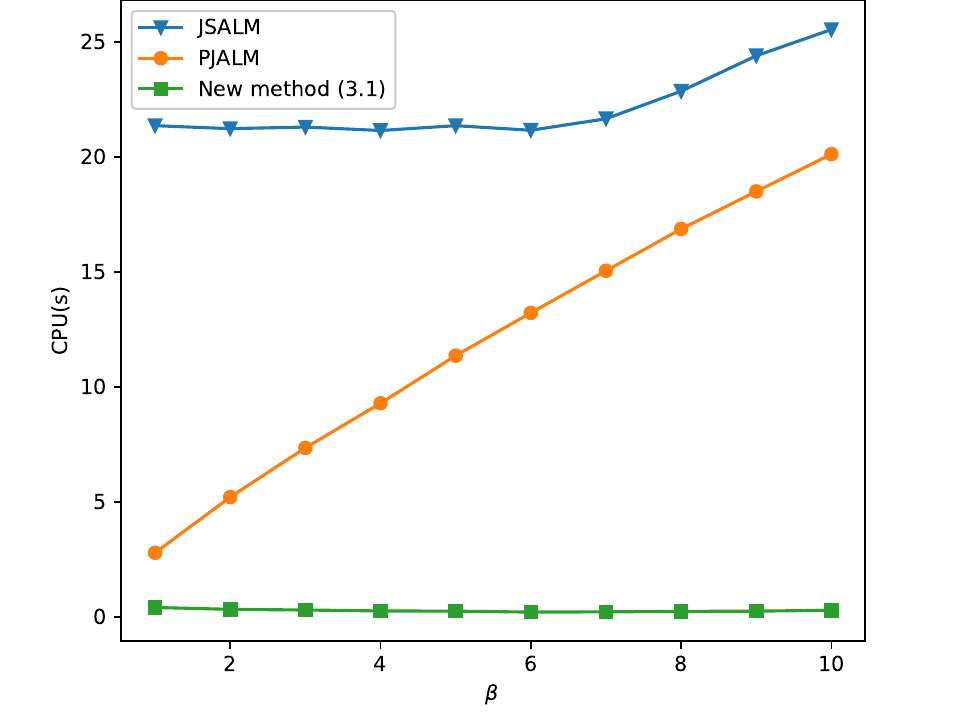}
}
\caption{Numerical results of \eqref{Exchange1} with $p=100$ with different $\beta \in[1,10]$. }
\label{fig5}
\end{figure}

\section{Extensions}\label{secE}
\setcounter{equation}{0}

To expose our idea more clearly, our discussion is focused on the model \eqref{A-Problem-M}. But the proposed new algorithm \eqref{Rank2M} and its theoretical analysis can also be extended to the following separable convex programming problem with linear inequality constraints:
\begin{equation}\label{A-Problem-CI}
  \min\Big\{\sum_{i=1}^p\theta_i(x_i)   \;\big|\;  \sum_{i=1}^p A_ix_i\ge b; \; x_i\in {\cal X}_i, \; i=1,\ldots, p\Big\},
\end{equation}
in which the settings are the same as those in \eqref{A-Problem-M}.

For this extension, we only need to modify the $\lambda$-subproblem in the parallel splitting ALM step \eqref{Pall-XL} as
\begin{equation}\label{inequlemad}
  \tilde{\lambda}^{k} = \big[\lambda^k - \textstyle\beta (\sum_{i=1}^p A_i x_i^{k}-b )\big]_+,
\end{equation}
where $[\cdot]_+$ denotes the standard projection operator onto the non-negative orthant in Euclidean space. The extended scheme of  \eqref{Rank2M} for \eqref{A-Problem-CI} thus reads as
 \begin{subequations} \label{Rank2M-I}
\begin{numcases}{}
\label{Pall-XL-I}\left\{
    \begin{array}{ll}
    \tilde{x}_1^{k} \; \in \; \arg\min\big\{  L(x_1,x_2^k, \ldots, x_p^k, \lambda^k)  +\frac{\beta}{2} \|A_1x_1-A_1x_1^k\|^2  \mid x_1\in\mathcal{X}_1  \big\},\\
       \quad \qquad \qquad \vdots \\
   \tilde{x}_i^{k}  \; \in \; \arg\min\big\{ L(x_1^k,\ldots,x_{i-1}^k, x_i, x_{i+1}^k, \ldots, x_p^k, \lambda^k)   +\frac{\beta}{2} \|A_ix_i-A_ix_i^k\|^2   \mid   x_i \in\mathcal{X}_i \big\},\\
       \quad \qquad  \qquad \vdots \\
   \tilde{x}_p^{k} \; \in \; \arg\min \big\{ L(x_1^k,\ldots, x_{p-1}^k, x_p, \lambda^k)  +\frac{\beta}{2} \|A_px_p-A_px_p^k\|^2  \mid   x_p \in\mathcal{X}_p \big\},\\[0.2cm]
   \tilde{\lambda}^{k} \; = \; [\lambda^k - \beta(\sum_{i=1}^p A_i x_i^k-b)]_+,
    \end{array}
      \right. \\[0.2cm]
\label{Correction-I} \;\; \xi^{k+1} \,=\; \xi^k-\alpha \mathcal{M}(\xi^k-\tilde{\xi}^k)\; \;\hbox{with} \;\; \alpha \in (0,2).
\end{numcases}
\end{subequations}

To establish the convergence analysis of the extended scheme \eqref{Rank2M-I} for \eqref{A-Problem-CI},  we only need to reset
$\Omega:=\mathcal{X}_1\times\mathcal{X}_2\times\cdots\times\mathcal{X}_p\times\Re_+^m$, and then modify \eqref{Full-pl} in Lemma \ref{VI-Lemma} as
$$
  \tilde{\lambda}^k \in \Re_+^m, \quad   (\lambda-\tilde{\lambda}^k)^T\Big[(\sum_{i=1}^p A_i \tilde{x}_i^{k}-b) -\sum_{i=1}^p (A_i\tilde{x}_i^{k}-A_ix_i^k)+\frac{1}{\beta}(\tilde{\lambda}^{k} - \lambda^k)\Big]\geq0, \quad \forall \; \lambda\in \Re_+^m,
$$
which is essentially the induced VI of the corresponding $\lambda$-subproblem \eqref{inequlemad}. The remaining proof can be seamlessly repeated  by the analysis in Sections \ref{sec4} and \ref{Sec-rate}.

\section{Conclusions}\label{sec6}
\setcounter{equation}{0}

In this paper, we present a rank-two relaxed parallel splitting version of the augmented Lagrangian method (ALM) for the multiple-block separable convex programming problem with linear equality constraints. The new algorithm adjusts the direct parallel splitting version of the ALM by both proximal regularization and relaxation techniques. Despite that the subproblems of the new algorithm are of the same difficulty as those of the existing algorithms of the same kind, the new algorithm requires no additional condition (in particular, no condition depending on the number of separable blocks in the model) while it maintains a step size in $(0,2)$ for further relaxing the primal and dual variables to ensure its convergence. We numerically validate the efficiency of the new algorithm with some application problems, and verify its significant acceleration over other existing algorithms of the same kind when the number of blocks in the separable model is large.

\end{CJK*}

\begin{thebibliography}{11}

\bibitem{ABMS-1}
Andreani, R., Birgin, E.G., Mart\'{\i}nez, J.M., Schuverdt, M.L.: On augmented Lagrangian methods with general lower-level constraints. SIAM J. Optim. \textbf{18}(4), 1286--1309 (2008)

\bibitem{ABMS-2}
Andreani, R., Birgin, E.G., Mart\'{\i}nez, J.M., Schuverdt, M.L.: Augmented Lagrangian methods under the constant positive linear dependence constraint qualification. Math. Program. \textbf{111}(1), 5--32 (2008)

\bibitem{Beck}
Beck, A.: First-Order Methods in Optimization, vol. 25. SIAM, Philadelphia (2017)

\bibitem{Bert1996}
Bertsekas, D.P.: Constrained Optimization and Lagrange Multiplier Methods. Athena Scientific, Belmont, MA (1996)

\bibitem{Birgin2014}
Birgin, E.G.,  Mart\'{\i}nez, J.M.: Practical Augmented Lagrangian Methods for Constrained Optimization. SIAM (2014)


\bibitem{boyd2010distributed}
Boyd, S., Parikh, N., Chu, E., Peleato, B., Eckstein, J.: Distributed optimization and statistical learning via the alternating direction method of multipliers. Found. Trends Mach. Learn. \textbf{3}(1), 1--122 (2010)


\bibitem{Can2011}
Cand\`{e}s, E.J., Li, X., Ma, Y.,  Wright, J.: Robust principal component analysis?  J. ACM. \textbf{58}(3), 11:1--37 (2011)


\bibitem{Can2013}
Cand\`{e}s, E.J., Sing-Long, C.A., Trzasko, J.D.: Unbiased risk estimates for singular value thresholding and spectral estimators. IEEE Trans. Signal Process. \textbf{61}(19), 4643--4657 (2013)



\bibitem{CP-Acta}
Chambolle, A., Pock, T.: An introduction to continuous optimization for imaging.  Acta Numer. \textbf{25}, 161--319 (2016)




\bibitem{chandrasekaran2012latent}
Chandrasekaran, V.,  Parrilo, P.A., Willsky, A.S.: Latent variable graphical model selection via convex optimization. Ann. Statist. \textbf{40},  1935--1967 (2012)


\bibitem{Chen}
Chen, S.S., Donoho, D.L.,  Saunders, M.A.: Atomic decomposition by basis pursuit. SIAM Rev. \textbf{43}, 129--159  (2001)


\bibitem{CHYY}
Chen, C.H., He, B.S., Ye, Y.Y., Yuan, X.M.: The direct extension of ADMM for multi-block convex minimization problems is not necessary convergent. Math. Program. \textbf{155}, 57--79 (2016)


\bibitem{Deng2017}
Deng, W., Lai, M.J., Peng, Z., Yin, W.: Parallel multi-block ADMM with $O(1/k)$ convergence. J. Sci. Comput. \textbf{71}, 712--736 (2017)


\bibitem{Eck94}
Eckstein, J.: Parallel alternating direction multiplier decomposition of convex programs. J. Optim. Theory Appl. \textbf{80}, 39--62 (1994)


\bibitem{Eck92}
Eckstein, J., Bertsekas, D.P.: On the Douglas-Rachford splitting method and the proximal points algorithm for maximal monotone operators. Math. Program. \textbf{55}, 293--318 (1992)




\bibitem{FHLY}
Fang, E.X., Liu, H., He, B.S., Yuan, X.M.: The generalized alternating direction method of multipliers: new theoretical insights and applications. Math. Prog. Comput. \textbf{7}, 149--187 (2015)


\bibitem{Fortin1983}
Fortin, M.,  Glowinski, R.: Augmented Lagrangian methods: Applications to the Numerical Solution of Boundary-Value Problems. Elsevier, Stud. Math. Appl. 15, North-Holland, Amsterdam (1983)




\bibitem{Glow84}
Glowinski, R.: Numerical Methods for Nonlinear  Variational Problems. Springer-Verlag, New York (1984)


\bibitem{GM}
Glowinski, R., Marrocco, A.: Approximation par $\acute{e}$l$\acute{e}$ments finis d'ordre un et r$\acute{e}$solution par p$\acute{e}$nalisation-dualit$\acute{e}$ d'une classe de probl$\grave{e}$mes non lin$\acute{e}$aires.  RAIRO Anal. Numer. R2, 41--76 (1975)


\bibitem{Glowinski89}
Glowinski, R., Tallec, P. Le.:  Augmented Lagrangian and Operator-Splitting Methods in Nonlinear Mechanics. SIAM, Philadelphia (1989)


\bibitem{Gol1979}
Gol'shtein, E.G., Tret'yakov, N.V.: Modified Lagrangians in convex programming and their generalizations. Math. Progr. Study. \textbf{10}, 86--97 (1979)




\bibitem{HeHouYuanSIOPT}
He, B.S., Hou, L.S., Yuan, X.M.: On full Jacobian decomposition of the augmented Lagrangian method for separable convex programming. SIAM J. Optim. \textbf{25},  2274--2312 (2015)


\bibitem{Hemyuan2020}
He, B.S., Ma, F., Yuan, X.M.: Optimal proximal augmented Lagrangian method and its application to full Jacobian splitting for multi-block separable convex minimization problems.  IMA J. Numer. Anal. \textbf{40}, 1188--1216 (2020)


\bibitem{he2012alternating}
He, B.S., Tao, M., Yuan, X.M.: Alternating direction method with Gaussian back substitution for separable convex programming. SIAM J. Optim. \textbf{22}(2), 313--340 (2012)


\bibitem{HeTaoYuanIMA}
He, B.S., Tao, M., Yuan, X.M.: A splitting method for separable convex programming. IMA J. Numer. Anal. \textbf{31}, 394--426 (2015)


\bibitem{he2017convergence}
He, B.S., Tao, M., Yuan, X.M.: Convergence rate analysis for the alternating direction method of multipliers with a substitution procedure for separable convex programming. Math. Oper. Res. \textbf{42}(3), 662--691 (2017)


\bibitem{HXY2016}
He, B.S., Xu, H.K., Yuan, X.M.: On the proximal Jacobian decomposition of ALM for multiple-block separable convex minimization problems and its relationship to ADMM. J. Sci. Comput. \textbf{66}, 1204--1217 (2016)


\bibitem{Heuniform2021}
He, B.S., Xu, S.J., Yuan, X.M.: Extensions of ADMM for separable convex optimization problems with linear equality or inequality constraints. arXiv preprint. arXiv:2107.01897 (2021)


\bibitem{HY-SINUM}
He, B.S., Yuan, X.M.: On the $O(1/n)$ convergence rate of Douglas-Rachford alternating direction method. SIAM J. Numer. Anal. \textbf{50}, 700--709 (2012)


\bibitem{HY-NM}
He, B.S., Yuan, X.M.: On non-ergodic convergence rate of Douglas-Rachford alternating directions method of multipliers. Numer. Math. \textbf{130}, 567--577 (2015)




\bibitem{Hes69}
Hestenes, M.R.:  Multiplier and gradient methods. J. Optim. Theory Appli. \textbf{4}, 303--320 (1969)



\bibitem{Ito2008}
Ito, K., Kunisch, K.: Lagrange Multiplier Approach to Variational Problems and Applications. Monographs and Studies in Mathematics, vol. 24. SIAM, Philadelphia (2008)


\bibitem{kiwiel1999proximal}
Kiwiel, K.C., Rosa, C.H., Ruszczynski, A.: Proximal decomposition via alternating linearization. SIAM J. Optim.  \textbf{9}(3), 668--689 (1999)




\bibitem{Mar70}
Martinet, B.: Regularisation, d'in\'equations variationelles par approximations succesives. Rev. Francaise d'Inform. Recherche Oper. \textbf{4},  154--159 (1970)


\bibitem{McLachlan}
McLachlan, G.J.: Discriminant Analysis and Statistical Pattern Recognition, vol. 544. Wiley Interscience, New York (2004)




\bibitem{Parikh2014}
Parikh, N., Boyd, S.: Proximal algorithms. Found. Trends Optim. \textbf{1}(3),  127--239 (2014)


\bibitem{PGWM}
Peng, Y.G., Ganesh, A., Wright, J., Xu, W.L., Ma, Y.: Robust alignment by sparse and low-rank decomposition for linearly correlated images. IEEE Trans. Pattern Anal. Mach. Intel. \textbf{34}, 2233--2246 (2012)


\bibitem{Powell69}
Powell, M.J.D.: A method for nonlinear constraints in minimization problems. In: Fletcher, R. (ed.) Optimization, pp. 283--298. Academic Press, New York (1969)


\bibitem{Rock76B}
Rockafellar, R.T.:  Augmented Lagrangians and applications of the proximal point algorithm in convex programming. Math. Oper. Res. \textbf{1},  97--116 (1976)


\bibitem{tao2011recovering}
Tao, M., Yuan, X.M.: Recovering low-rank and sparse components of matrices from incomplete and noisy observations. SIAM J. Optim. \textbf{21}(1), 57--81 (2011)


\bibitem{TaoYuan2018}
Tao, M., Yuan, X.M.: On the optimal linear convergence rate of a generalized proximal point algorithm. J. Sci. Comput. \textbf{74}(2), 826--850 (2018)


\bibitem{TY-JOTA}
Tao, M., Yuan, X.M.: On Glowinski's open question on the alternating direction method of multipliers. J. Optim. Theory Appli. \textbf{179}, 163--196 (2018)

\bibitem{Wen}
Wen, Z., Goldfarb, D., Yin, W.: Alternating direction augmented Lagrangian methods for semidefinite programming. Math. Prog. Comput. \textbf{2}, 203--230 (2010)


\end{thebibliography}
\end{document}